\documentclass[11pt,reqno]{amsart}
\usepackage{a4wide}
\usepackage[utf8]{inputenc}
\usepackage{amsmath, mathscinet, amsthm, amscd, amsfonts, enumitem, amssymb, graphicx, xcolor, mathrsfs, dsfont}
\usepackage[english]{babel}
\usepackage[colorlinks=true,linkcolor=blue]{hyperref}

\makeatletter
\newcommand{\myitem}[2]{
  \item[#1]
  \edef\@currentlabel{#1}
  \label{#2}}
\makeatother

\numberwithin{equation}{section}

\theoremstyle{plain}
\newtheorem{theorem}{Theorem}[section]
\newtheorem{proposition}[theorem]{Proposition}
\newtheorem{lemma}[theorem]{Lemma}
\newtheorem{corollary}[theorem]{Corollary}

\theoremstyle{definition}
\newtheorem{definition}[theorem]{Definition}

\theoremstyle{remark}
\newtheorem{remark}[theorem]{Remark}
\newtheorem{example}[theorem]{Example}

\newcommand{\al}{\alpha}

\newcommand{\de}{\delta}
\newcommand{\ep}{\varepsilon}
\newcommand{\ph}{\varphi}
\newcommand{\ga}{\gamma}
\newcommand{\om}{\omega}
\newcommand{\rh}{\varrho}
\newcommand{\si}{\sigma}
\renewcommand{\th}{\theta}

\newcommand{\Om}{\Omega}

\newcommand{\dist}{\operatorname{dist}}
\newcommand{\trace}{\operatorname{tr}}
\newcommand{\supp}{\operatorname{supp}}

\renewcommand{\d}{\text{\rm d}}
\newcommand{\eins}{\mathds{1}}

\newcommand{\Cnt}{\operatorname{C}}

\newcommand{\Lip}{\operatorname{Lip}}

\newcommand{\cmp}{{\rm c}}
\newcommand{\bdd}{{\rm b}}

\newcommand{\Cb}{\Cnt_\bdd}
\newcommand{\Cc}{\Cnt_\cmp}
\newcommand{\Cbinf}{\Cnt_\bdd^\infty}
\newcommand{\Ccinf}{\Cnt_\cmp^\infty}
\newcommand{\Lipb}{\Lip_\bdd}
\newcommand{\BUC}{\operatorname{UC}_\bdd}

\newcommand{\E}{\mathbb{E}}
\newcommand{\N}{\mathbb{N}}
\renewcommand{\P}{\mathbb{P}}

\newcommand{\R}{\mathbb{R}}

\newcommand{\Bc}{\mathcal{B}}

\newcommand{\Pc}{\mathcal{P}}

\begin{document}

\title[Chernoff-Mehler Formula for L\'evy Processes with Drift]{Chernoff-Mehler  Approximation for L\'evy Processes with Drift}

\author{Max Nendel}
\address{Department of Statistics and Actuarial Science, University of Waterloo}
\email{mnendel@uwaterloo.ca}

\thanks{The author thanks Jonas Blessing, Sven Fuhrmann, Markus Kunze, and Michael Kupper for valuable comments and discussions related to this work.}
\date{\today}

\begin{abstract}
	In this paper, we study an approximation scheme for L\'evy processes with drift in terms of a representation that is akin to the celebrated Mehler formula for L\'evy-Ornstein-Uhlenbeck processes.\ The approximation scheme is based on a variant of the Chernoff product formula on the space of bounded continuous functions. In a first step, we provide sufficient and necessary conditions for arbitrary families of probability measures, indexed by positive real numbers, to give rise to a convolution semigroup via a Chernoff approximation on the space of bounded continuous functions, equipped with the mixed topology.\ In this context, we provide explicit criteria both for the convergence of subsequences and the entire family, and discuss fine properties related to the domain of the associated generator of the L\'evy process and the infinitesimal behavior of the approximating family of measures.\ In a second step, we enrich the family of measures by a deterministic component and derive explicit conditions that ensure both the convergence of subsequences and the entire family to a L\'evy process with drift under a Chernoff approximation.\ In a series of examples, we show that our general conditions on the dynamics are satisfied, for example, by flows of Lipschitz ordinary differential equations, Euler schemes, and arbitrary Runge-Kutta methods, and that the Central Limit Theorem can be subsumed under our framework. \medskip
    
	 \smallskip
	
\end{abstract}

\maketitle

\section{Introduction}

In this paper we discuss a specific version of the classical Chernoff product formula, cf.\ Engel and Nagel \cite[Section III.5]{engel2000one} and Pazy \cite[Section 3.5]{pazy2012semigroups} for L\'evy processes with drift on $\R^d$.\ We consider an arbitrary family of Borel probability measures $(\mu_t)_{t>0}$ and a family $(\psi_t)_{t>0}$ of measurable functions $\R^d\to \R^d$, and define the transition operator
\begin{equation}\label{eq.transition.intro}
 (P_t f)(x):=\int_{\R^d} f\big(\psi_t(x)+y\big)\,\mu_t(\d y)\quad \text{for }f\in \Cb(\R^d),
\end{equation}
 where $\Cb(\R^d)$ denotes the space of all bounded continuous functions $\R^d\to \R$. If $\psi_t(x)=x$ and $\mu_t=\P\circ Y_t^{-1}$ for all $t>0$ and $x\in \R^d$ with a L\'evy process $(Y_{t})_{t\geq 0}$ on a probability space $(\Om,\mathcal F,\P)$, then $(P_t)_{t\geq 0}$ is simply the transition semigroup of the L\'evy process $(Y_{t})_{t\geq 0}$. If $\psi_t(x)=e^{tA}x$ for some matrix $A\in \R^{d\times d}$ and $\mu_t$ is the distribution of the stochastic convolution $\int_0^t e^{(t-s)A}\,\d (\sigma W_s+bs)$ for all $t>0$ and $x\in \R^d$ with $b\in \R^d$, $\si\in \R^{d\times d}$ symmetric and positive semidefinite and a standard Brownian motion $(W_t)_{t\geq 0}$ on $\R^d$, \eqref{eq.transition.intro} resembles the celebrated Mehler formula for the transition semigroup of the Ornstein-Uhlenbeck process, given by the SDE
 \[
   \d X_t= (AX_t+b)\, \d t+\si\, \d W_t.
 \]
 We refer to Bochagev et al.\ \cite{zbMATH00897938} and Fuhrman and R\"ockner \cite{zbMATH01436912} for generalizations of the classical Mehler formula to infinite-dimensional and non-Gaussian settings, respectively.\ We also point out that \eqref{eq.transition.intro} is also somewhat akin to the so-called Trotter-Lie product formula or splitting method for partial differential equations, cf.\ Engel and Nagel \cite[Corollary III.5.8]{engel2000one} and Zagrebnov et al.\ \cite[Proposition 1.9.35]{zbMATH07851043}.
 
 In this paper, we investigate (refined versions of) the following two main questions related to the families $(\mu_t)_{t>0}$ and $(\psi_t)_{t>0}$:
 \begin{itemize}
\item What are sufficient (and necessary) conditions, for arbitrary families $(\mu_t)_{t>0}$ and $(\psi_t)_{t>0}$, for the Chernoff product $(P_{t/n}^n f)_{n\in \N}$ to converge for all $t>0$ and $f\in \Cb(\R^d)$?
\item Which are the limits that arise from the Chernoff product or, more precisely, can we obtain an explicit characterization of the limiting dynamics?
 \end{itemize}
 A major challenge in the context of the classical Chernoff approximation is the identification of a suitable core for the generator of the limiting dynamics.\ While this is achievable for the L\'evy case, i.e., $\psi_t(x)=x$ for all $t>0$ and $x\in \R^d$, it poses a major challenge in the case of a general family $(\psi_t)_{t>0}$.\ Another difficulty lies in the fact that, already for the case of an Ornstein-Uhlenbeck process, it is known that the transition semigroup is not strongly continuous on the space of bounded uniformly continuous functions with the topology of uniform convergence, see, for instance, Tessitore and Zabczyk \cite{zbMATH01626172} for a discussion in the framework of Trotter-Lie product formulae for transition semigroups.\ In our analysis, we therefore adopt a recently developed variant of the classical Chernoff approximation, cf.\ Blessing et al.\ \cite{blessing2022convex} and Blessing et al.\ \cite{blessing2022convergence}, on the space $\Cb(\R^d)$ with the mixed topology, cf.\ Wiweger \cite{wiweger} and Fremlin et al.\ \cite{haydon} as well as Kunze \cite{kunze}, Goldys and Kocan \cite{kocan}, and Goldys et al.\ \cite{Goldys2022mixed} for surveys on semigroups in mixed or strict topologies and their relation to Markov processes, which fall into the more general framework of so-called bi-continuous semigroups.\ In this context, we refer to K\"uhnemund \cite{kuhnemund2001} for a general theory and approximation results for
bi-continuous semigroups, Albanese and K\"uhnemund \cite{zbMATH01876814} for Trotter–Kato and Lie–Trotter product formulae for locally equicontinuous semigroups on sequentially complete locally convex spaces, and Albanese and Mangino \cite{zbMATH02057702} for a Trotter-Kato theorem and Chernoff approximation for bi-continuous semigroups.\ Moreover, we refer to K\"uhnemund and van Neerven \cite{zbMATH02111542} for a Lie-Trotter product formula for the Ornstein-Uhlenbeck semigroup on real separable Banach spaces, Butko \cite{zbMATH06871300} for an overview of Chernoff approximations for Feller semigroups on the space of all continuous functions vanishing at infinity, including a particular representation of the approximating family as a pseudo-differential operator w.r.t.\ the symbol of the differential operator in the spirit of the method of ``freezing the coefficients'' as well as approximations for multiplicative dynamics, and Butko \cite{zbMATH07238029} for Chernoff approximations of subordinate semigroups.

In this work, we consider teh following three conditions:
\begin{itemize}
\item a boundedness condition \eqref{cond.M} that
\[
 \sup_{h\in (0,h_0)}\frac1h\Bigg( \int_{\R^d} {1\wedge |y|^2}\,\mu_h(\d y)+\bigg|\int_{\{|y|\leq 1\}} y\,\mu_h(\d y)\bigg|\Bigg)<\infty\quad \text{for some }h_0>0,
\]
\item a tightness condition \eqref{cond.T} that, for all $\ep>0$, there exists $M_\ep>0 $ such that
	   \[
	       \limsup_{h\downarrow0}\frac{\mu_h\big(\big\{y\in \R^d\,\big|\, |y|>M_\ep\big\}\big)}h< \ep,
           \]
           \item a uniform Lipschitz condition \eqref{cond.D} that there exist $\omega\geq 0$, $\delta>0$, and $h_0>0$ such that, for all $u\in \R^d$ with $|u|\leq \delta$,
	\[
		\sup_{h\in (0,h_0)}\sup_{x\in \R^d}\frac{|\psi_h(x+u)-\psi_h(x)-u|}{h}\leq \omega |u|,
	\]
    and $\limsup_{h\downarrow0}\frac{|\psi_h(0)|}{h}<\infty$. 
\end{itemize}
Our first main result, Theorem \ref{thm.main.mehler}, then shows that the conditions \eqref{cond.M}, \eqref{cond.T}, and \eqref{cond.D} guarantee that every null sequence in $(0,\infty)$ has a further subsequence $(h_n)_{n\in \N}$ such that $P_{h_n}^{k_n}f_n$ converges, uniformly on compacts, to the transition semigroup of a L\'evy process with drift at time $t=\lim_{n\to \infty} k_nh_n\in [0,\infty)$, evaluated in $f=\lim_{n\to \infty} f_n\in \Cb(\R^d)$ (in the mixed topology).\ This result is based on a detailed analysis of the conditions \eqref{cond.M}, \eqref{cond.T}, and \eqref{cond.D} in Section \ref{sec.M}, Section \ref{sec.T}, and Section \ref{sec.D}, respectively. In particular, adopting techniques from the proof of Courr\`ege's theorem, see, e.g., \cite[Theorem 4.5.21]{jacob2001_volume_I}, we show that Condition \eqref{cond.M} is equivalent to a compactness criterion in the sense that, for every null sequence in $(0,\infty)$, there exist a L\'evy tiplet $(b,\si,\nu)$, a constant $c\geq 0$, and a subsequence $(h_n)_{n\in \N}$ such that, for every $f\in \Ccinf(\R^d)$ and $x\in \R^d$, the difference quotient
\[
\int_{\R^d} \frac{f(x+y)-f(x)}{h_n}\,\mu_{h_n}(\d y)
\]
converges pointwise to the generator of a killed L\'evy process with L\'evy triplet $(b,\si,\nu)$ and killing $c$, see Theorem \ref{prop.ex.levytriplet}, which is the second main result.

We point out that the conditions \eqref{cond.M}, \eqref{cond.T}, and \eqref{cond.D}, and even their stronger variants \eqref{cond.Mstar} and \eqref{cond.Dstar} that ensure convergence of the Chernoff-Mehler approximation for every null sequence in $(0,\infty)$ to the same limit, see Corollary \ref{thm.main.mehler.stronger}, are satisfied for the following examples:
\begin{itemize}
    \item transition probabilities of L\'evy processes starting in zero, see Example \ref{ex.levy},
    \item transition probabilities of stochastic convolutions starting in zero, see Example \ref{ex.OU},
    \item stochastic finite difference approximations for second-order differential operators, leading to the Central Limit Theorem, see Example \ref{ex.clt},
    \item flows for Lipschitz ODEs, see Example \ref{ex.diffusion},
    \item  Euler schemes and general Runge-Kutta methods for Lipschitz ODEs, see Example \ref{ex.euler} and Example \ref{ex.runge}, respectively.
\end{itemize}

The rest of the paper is organized as follows.\ In Section \ref{sec.setup}, we introduce the notation and present the first main result, Theorem \ref{thm.main.mehler} and its corollaries. In Section \ref{sec.M}, we discuss equivalent formulations of Condition \eqref{cond.M}, including the second main result Theorem \ref{prop.ex.levytriplet}. Section \ref{sec.T} examines equivalences for the combination of Condition \eqref{cond.M} and Condition \eqref{cond.T}. Section \ref{sec.D} studies the interplay between the conditions \eqref{cond.M}, \eqref{cond.T}, and \eqref{cond.D}.\ The proofs of the main result and its corollaries are contained in Section \ref{sec.proofs.main}. Appendix \ref{app.C} discusses the existence of suitable smooth cut-off functions, which play an integral role in the proof of the first main result, Appendix \ref{app.semigroup} provides an abstract approximation result for the infinitesimal behavior of families of operators on Banach spaces, and, in Appendix \ref{app.levydrift}, we prove some auxiliary results for L\'evy processes with drift.

\section{Setup and First Main Result}\label{sec.setup}
Throughout, let $d\in \N$ and $\R^d$ be endowed with the Euclidean norm $|\cdot|$, the standard inner product $\langle\cdot,\cdot\rangle$, and the Borel $\sigma$-algebra $\Bc(\R^d)$.\ We denote the set of all probability measures on $\Bc(\R^d)$  by $\Pc(\R^d)$, and we consider the space $\Cb(\R^{d})$ of all bounded and continuous functions $\R^d\to\R$, endowed with the mixed topology, cf.\ \cite{haydon, wiweger}.\ We point out that the mixed topology is the Mackey topology of the dual pair $(\Cb,{\rm ca})$, where ${\rm ca}$ denotes the space of all countably additive signed Borel measures on $\R^d$ of finite variation.\

In this work, we restrict our attention to sequential properties related to the mixed topology.\ Mainly because, for positive linear functionals, continuity in the mixed topology is equivalent to sequential continuity and even continuity from above or below, cf.\ \cite{Nendel2022lsc} for a discussion in a more general context.\ For a sequence $(f_n)_{n\in\N}\subset \Cb(\R^d)$ and  $f\in \Cb(\R^d)$, $f=\lim_{n\to \infty} f_n$ in the mixed topology if
\begin{equation}\label{eq.conv.mixed}
\sup_{n\in\N} \|f_n\|_\infty<\infty\quad\text{and}\quad \lim_{n\to\infty}\sup_{|x|\le r}|f_n(x)-f(x)|=0\quad\text{for all }r\ge 0. 
\end{equation}
Here and throughout, $\|f\|_\infty:=\sup_{x\in\R^d}|f(x)|$ denotes the supremum norm of $f\in \Cb(\R^d)$. Analogously, for a family $(f_t)_{t>0}\subset \Cb(\R^d)$ and $f\in \Cb(\R^{d})$, we write $f=\lim_{h\downarrow 0}f_h$ in the mixed topology if $f=\lim_{n\to \infty}f_{h_n}$ in the mixed topology for every null sequence $(h_n)_{n\in \N}\subset (0,\infty)$.

In addition to $\Cb(\R^d)$, we also consider the following function spaces:
\begin{itemize}
    \item the space $\Cc(\R^d)$ of all $f\in \Cb(\R^d)$ with compact support and its closure $\Cnt_0(\R^d)$ w.r.t.\ the supremum norm,
    \item the space $\BUC(\R^d)$ of all bounded uniformly continuous functions $\R^d\to \R$,
    \item the spaces $\Cb^k(\R^d)$, $\Cc^k(\R^d)$, $\Cnt_0^k(\R^d)$, and $\BUC^k(\R^d)$ of all functions $f\in \Cb(\R^d)$ with $\partial^\alpha f\in \Cb(\R^d)$, $\partial^\alpha f\in \Cc(\R^d)$, $\partial^\alpha f\in \Cnt_0(\R^d)$, and $\partial^\alpha f\in \BUC(\R^d)$ for all $\al\in \N_0^d$ with $|\al|\leq k$, respectively,
    \item $\Cbinf(\R^d):=\bigcap_{k\in \N}\Cb^k(\R^d)$ and $\Ccinf(\R^d):=\bigcap_{k\in \N}\Cc^k(\R^d)$,
    \item the space $\Lipb(\R^d)$ of all functions $f\in \Cb(\R^d)$ with
\begin{equation}\label{eq.lip} \|f\|_{\Lip}:=\sup_{x_1\neq x_2}\frac{|f(x_1)-f(x_2)|}{|x_1-x_2|}<\infty, \end{equation}
    \item The space $L_p(\R^d)$ of all (equivalence classes of) Borel measurable functions $f\colon \R^d\to \R$ with
\[
\|f\|_p:=\bigg(\int_{\R^d} |f(x)|^p\,\d x\bigg)^{1/p}<\infty,
\]
for $p\in [1,\infty)$, and the space $L_\infty(\R^d)$ of all (equivalence classes of) bounded Borel measurable functions $f\colon \R^d\to \R$, endowed with the (essential) supremum norm $\|\cdot\|_{\infty}$,
\item the Sobolev space $W_p^k(\R^d)$, for $p\in [1,\infty]$, of all $f\in L_p(\R^d)$ with weak derivative $\partial^\alpha f\in L_p(\R^d)$ for all $\al\in \N_0^d$ with $|\al|\leq k$.
\end{itemize}

For a L\'evy triplet $(b,\sigma,\nu)$, we define $L_{(b,\sigma,\nu)}\colon \Cb^2(\R^d)\to \Cb(\R^d)$ by
\begin{equation}\label{eq.def.levygenerator}
\big(L_{(b,\sigma,\nu)} f\big)(x)=\big\langle b,\nabla f(x)\big\rangle+\frac12\trace\big(\sigma \nabla^2f(x)\big)+\int_{\R^d\setminus\{0\}} f(x+y)-f(x)-\big\langle\nabla f(x),h(y)\big\rangle \,\nu(\d y)
\end{equation}
for all $f\in \Cb^2(\R^d)$ and $x\in \R^d$ with  $h(y):=y\eins_{[0,1]}\big(|y|\big)$ for all $y\in \R^d$.\ Moreover, for a L\'evy triplet $(b,\sigma,\nu)$ and an $\om$-Lipschitz function $F\colon \R^d\to \R^d$ with $\om\geq 0$, i.e., 
		\begin{equation}\label{eq.Fgloballip}
			|F(x_1)-F(x_2)|\leq \omega|x_1-x_2|\quad\text{for all }x_1,x_2\in \R^d,
		\end{equation}
we consider the operator $A_{F,(b,\sigma,\nu)}\colon D(A_{F,(b,\sigma,\nu)})\to \Cb(\R^d)$, given by
\[
 D(A_{F,(b,\sigma,\nu)}):=\big\{f\in \Cb^2(\R^d)\,\big|\, \langle \nabla f,F\rangle \in \Cb(\R^d)\big\}
\]
and
\[
\big(A_{F,(b,\sigma,\nu)} f\big)(x)=\big\langle \nabla f(x), F(x)\big\rangle+\big(L_{(b,\sigma,\nu)} f\big)(x)\quad \text{for all }f\in D(A_{F,(b,\sigma,\nu)})\text{ and }x\in \R^d.
\]

Throughout, let $(\mu_t)_{t>0}\subset \Pc(\R^{d})$ be a family of probability measures on $\Bc(\R^d)$.\ We consider the following condition:
\begin{enumerate}
    \item[(M)]\makeatletter\def\@currentlabel{M}\makeatother\label{cond.M} There exists $h_0>0$ such that
	   \begin{equation}
           \notag
	       \sup_{h\in (0,h_0)}\frac1h\Bigg( \int_{\R^d} {1\wedge |y|^2}\,\mu_h(\d y)+\bigg|\int_{\{|y|\leq 1\}} y\,\mu_h(\d y)\bigg|\Bigg)<\infty. 
	   \end{equation}
\end{enumerate} 
Moreover, we consider a deterministic flow, given in terms of a family $(\psi_t)_{t>0}$ of functions $\R^d\to\R^d$, satisfying the following condition:
\begin{enumerate}
	\item[(D)]\makeatletter\def\@currentlabel{D}\makeatother\label{cond.D} There exist $\omega\geq 0$, $\delta>0$, and $h_0>0$ such that, for all $u\in \R^d$ with $|u|\leq \delta$,
	\begin{equation}\label{eq.biton}
		\sup_{h\in (0,h_0)}\sup_{x\in \R^d}\frac{|\psi_h(x+u)-\psi_h(x)-u|}{h}\leq \omega |u|.
	\end{equation}   
    Moreover, $\limsup_{h\downarrow0}\frac{|\psi_h(0)|}{h}<\infty$. 
\end{enumerate}
Observe that the condition $\limsup_{h\downarrow0}\frac{|\psi_h(0)|}{h}<\infty$ can always be enforced by considering $\overline\psi_t(x):=\psi_t(x)-\psi_t(0)$ instead of $\psi_t(x)$ for all $t>0$ and $x\in \R^d$.\ Moreover, we point out that there is no relation between the conditions \eqref{cond.M} and \eqref{cond.D}, so that the choice of the family $(\psi_t)_{t>0}$ is completely independent of the family $(\mu_t)_{t>0}$.\ In particular, the choices $\mu_t=\de_0$ or $\psi_t(x)=x$ for all $t>0$ and $x\in \R^d$ are always admissible, and correspond to the cases of Koopman semigroups or transition semigroups of L\'evy processes, respectively.

In addition to the conditions \eqref{cond.M} and \eqref{cond.D}, we also consider the following tightness condition for the family $(\mu_t)_{t>0}$:
 \begin{enumerate}
	\item[(T)]\makeatletter\def\@currentlabel{T}\makeatother\label{cond.T} For all $\ep>0$, there exists $M_\ep>0 $ such that
	   \begin{equation}\label{eq.Ass.M2}
	       \limsup_{h\downarrow0}\frac{\mu_h\big(\big\{y\in \R^d\,\big|\, |y|>M_\ep\big\}\big)}h< \ep.
	   \end{equation}
 \end{enumerate}
For a detailed discussion of conditions \eqref{cond.M}, \eqref{cond.T}, and \eqref{cond.D}, we refer to Section \ref{sec.M}, Section \ref{sec.T}, and  Section \ref{sec.D}, respectively.

In the sequel, we consider the family $(P_t)_{t> 0}$ of operators $\Cb(\R^d)\to \Cb(\R^d)$, given by  
\begin{equation}\label{eq:reference}
	\big(P_tf\big)(x):=\int_{\R^d} f\big(\psi_t(x)+y\big)\,\mu_t(\d y)
\end{equation}
for $t> 0$, $f\in \Cb(\R^d)$, and $x\in \R^d$.\ 
The following theorem is the main result of this paper.

\begin{theorem}\label{thm.main.mehler}
	Assume that conditions \eqref{cond.M}, \eqref{cond.T}, and \eqref{cond.D} are satisfied.\ Then, for every null sequence in $(0,\infty)$, there exist a subsequence $(h_n)_{n\in \N}$, a L\'evy process $(Y_t)_{t\geq0}$ on some probability space $(\Omega,\mathcal F,\P)$, and an $\omega$-Lipschitz function $F\colon \R^d\to \R^d$ such that,
	for all $t\geq 0$, $f\in \Cb(\R^d)$, $(k_n)_{n\in \N}\subset \N$ with $k_nh_n\to t$ as $n\to \infty$, $(f_n)_{n\in \N}\subset \Cb(\R^d)$ with $f_n\to f$ in the mixed topology as $n\to \infty$, and $r\geq 0$,
	\[
	\sup_{|x|\leq r}\Big|\big(P_{h_n}^{k_n}f_n\big)(x)- \E_\P[f(X_t^x)]\Big|\to 0\quad \text{as }n\to \infty,
	\]
	where, for all $x\in \R^d$, the process $(X_t^x)_{t\geq 0}$ is the unique strong solution to the L\'evy SDE
	\begin{equation}\label{eq.levy.sde}
	\d X_t^x= F(X_t^x)\,\d t+\d Y_t\quad \text{with}\quad X_0^x=x.
	\end{equation}
	Moreover, if $f\in \Cb(\R^d)$ such that $\big(\frac{P_{h_n}f-f}{h_n}\big)_{n\in \N}$ converges in the mixed topology, then
    \[
    \lim_{h\downarrow0}\frac{\E_\P[f(X_{h}^\cdot)]-f}{h}=\lim_{n\to \infty}\frac{P_{h_n}f-f}{h_n}\quad \text{in the mixed topology.}
    \]
\end{theorem}

For $\mu\in \Pc(\R^d)$ and $k\in \N$, we use the notation
\[
\mu^{\ast k}:=\underbrace{\mu\ast\ldots\ast\mu}_{k\text{ times}}.
\]
A direct consequence of Theorem \ref{thm.main.mehler}, choosing $\psi_t(x)=x$ for all $t>0$ and $x\in \R^d$, is the following corollary.

\begin{corollary}\label{thm.main.levy}
    Assume that conditions \eqref{cond.M} and \eqref{cond.T} are satisfied.\ Then, for every null sequence in $(0,\infty)$, there exist a subsequence $(h_n)_{n\in \N}$ and a L\'evy process $(Y_t)_{t\geq0}$ on some probability space $(\Omega,\mathcal F,\P)$ such that
    \[
    \mu_{h_n}^{\ast k_n}\to \P\circ Y_t^{-1}\quad \text{in distribution as }n\to \infty
    \]
    for all sequences $(k_n)_{n\in \N}\subset \N$ with $k_nh_n\to t\in [0,\infty)$ as $n\to \infty$.
\end{corollary}

We now consider the following stronger versions of the conditions \eqref{cond.M} and \eqref{cond.D}:
\begin{enumerate}
\item[(M$^*$)]\makeatletter\def\@currentlabel{M$^*$}\makeatother\label{cond.Mstar} The limit
	\begin{equation}\label{eq.M1star}
	\lim_{h\downarrow 0} \int_{\R^d}\frac{\ph(y)-\ph(0)}{h}\, \mu_h(\d y)\in \R\quad \text{exists for all }\ph\in \Ccinf(\R^d).
	\end{equation}
    \item[(D$^*$)]\makeatletter\def\@currentlabel{D$^*$}\makeatother\label{cond.Dstar} In addition \eqref{eq.biton}, for all $x\in \R^d$,
	\[
	F(x):=\lim_{h\downarrow 0} \frac{\psi_h(x)-x}{h}\quad\text{exists in }\R^d.
	\]
\end{enumerate} 
We then have the following corollary.

\begin{corollary}\label{thm.main.mehler.stronger}
	Assume that conditions \eqref{cond.Mstar}, \eqref{cond.T}, and \eqref{cond.Dstar} are satisfied.\ Then, there exists a unique (in law) L\'evy process $(Y_t)_{t\geq0}$ on some probability space $(\Omega,\mathcal F,\P)$ such that, for every null sequence $(h_n)_{n\in \N}\subset (0,\infty)$, all $t\geq 0$, $f\in \Cb(\R^d)$, $(k_n)_{n\in \N}\subset \N$ with $k_nh_n\to t$ as $n\to \infty$, $(f_n)_{n\in \N}\subset \Cb(\R^d)$ with $f_n\to f$ in the mixed topology as $n\to \infty$, and $r\geq 0$,
	\[
	\sup_{|x|\leq r}\Big|\big(P_{h_n}^{k_n}f_n\big)(x)- \E_\P[f(X_t^x)]\Big|\to 0\quad \text{as }n\to \infty,
	\]
	where $(X_t^x)_{t\geq 0}$ is the unique strong solution to \eqref{eq.levy.sde} for all $x\in \R^d$.\
	If $f\in \Cb(\R^d)$ such that $\lim_{h\downarrow 0}\frac{P_{h}f-f}{h}$ exists in the mixed topology, then
    \[
     \sup_{h>0} \bigg\|\frac{\E_\P[f(X_{h}^\cdot)]-P_hf}{h}\bigg\|_\infty<\infty\quad\text{and}\quad
    \lim_{h\downarrow0}\sup_{|x|\leq r}\bigg|\frac{\E_\P[f(X_{h}^x)]-(P_hf)(x)}{h}\bigg|=0
    \]
    for all $r\geq0$.
\end{corollary}

Observe that, for all $f\in L_p(\R^d)$ and $\mu\in \Pc(\R^d)$,
\[
\int_{\R^d} f(\, \cdot+y)\,\mu(\d y)\in L_p(\R^d),
\]
see, e.g., Lemma \ref{lem.convolution}.\ This allows to understand $P_t$ as an operator $L_p(\R^d)\to L_p(\R^d)$ if $\psi_t(x)=x$ for all $t>0$ and $x\in \R^d$.

\begin{corollary}\label{thm.main.mehler.stronger.Lp}
	Assume that conditions \eqref{cond.Mstar} and \eqref{cond.T} are satisfied.\ Then, there exists a unique (in law) L\'evy process $(Y_t)_{t\geq0}$ on some probability space $(\Omega,\mathcal F,\P)$ such that
    \begin{enumerate}
    \item[(i)] for every null sequence $(h_n)_{n\in \N}\subset (0,\infty)$, all $p\in [1,\infty)$, $t\geq 0$, $f\in L_p(\R^d)$, $(k_n)_{n\in \N}\subset \N$ with $k_nh_n\to t$ as $n\to \infty$, and $(f_n)_{n\in \N}\subset L_p(\R^d)$ with $f_n\to f$ in $L_p(\R^d)$ as $n\to \infty$,
	\[
	\lim_{n\to \infty}\bigg\| \int_{\R^d} f(\, \cdot+y)\,\mu_{h_n}^{\ast k_n}(\d y)-\E_\P[f(\,\cdot +Y_t)]\bigg\|_p=0,
	\]
    \item[(ii)] for every null sequence $(h_n)_{n\in \N}\subset (0,\infty)$, all $t\geq 0$, $f\in \BUC(\R^d)$, $(k_n)_{n\in \N}\subset \N$ with $k_nh_n\to t$ as $n\to \infty$, and $(f_n)_{n\in \N}\subset \BUC(\R^d)$ with $f_n\to f$ w.r.t.\ $\|\cdot\|_\infty$ as $n\to \infty$,
	\[
	\lim_{n\to \infty}\bigg\| \int_{\R^d} f(\, \cdot+y)\,\mu_{h_n}^{\ast k_n}(\d y)-\E_\P[f(\,\cdot +Y_t)]\bigg\|_\infty=0.
	\]
    \end{enumerate}
    If $f\in L_p(\R^d)$ such that $\lim_{h\downarrow 0}\int_{\R^d}\frac{f(\, \cdot\, + y)-f}{h}\,\mu_h(\d y)\in L_p(\R^d)$ exists, then
    \[
     \lim_{h\downarrow0} \frac1h\bigg\| \int_{\R^d} f(\, \cdot+y)\,\mu_{h}(\d y)-\E_\P[f(\,\cdot +Y_h)]\bigg\|_p=0.
    \]
    If $f\in \BUC(\R^d)$ such that $\lim_{h\downarrow 0}\int_{\R^d}\frac{f(\, \cdot\, + y)-f}{h}\,\mu_h(\d y)\in \BUC(\R^d)$ exists, then
    \[
     \lim_{h\downarrow0} \frac1h\bigg\| \int_{\R^d} f(\, \cdot+y)\,\mu_{h}(\d y)-\E_\P[f(\,\cdot +Y_h)]\bigg\|_\infty=0.
    \]
\end{corollary}

\section{Discussion of Condition (M)}\label{sec.M}

In this section, we discuss Condition \eqref{cond.M} from different perspectives, providing equivalent formulations.\ We start with 
the following a priori estimate, which plays a central role in the subsequent discussion.

\begin{lemma}\label{lem.apriori.gamma}
Assume that Condition \eqref{cond.M} is satisfied.
\begin{enumerate}
\item[a)] For all $M>0$,
\begin{equation}\label{eq.M2.weaker}
c_M:=\sup_{t>0}\frac{\mu_t\big(\big\{ y\in \R^d\,\big|\, |y|>M\big\}\big)}t<\infty
\end{equation}
and
\begin{equation}\label{eq.Cgamma}
C_M:=\sup_{t>0}\frac1t\Bigg(\bigg| \int_{\{|y|\leq M\}}y\,\mu_t(\d y)\bigg|+\int_{\{|y|\leq M\}} |y|^2\,\mu_t(\d y)\Bigg)<\infty.
\end{equation}
\item[b)] For all $M>0$, $f\in \Cb^2(\R^d)$, $x\in \R^d$, and $t>0$,
\begin{align}
	\notag \bigg|&\int_{\R^d}\frac{f(x+y)-f(x)}{t}\,\mu_t(\d y)\bigg|\leq  \int_{\{|y|>M\}}\frac{|f(x+y)-f(x)|}t\,\mu_t(\d y)\\
	\notag &\qquad \quad+\frac1t\Bigg( \big|\nabla f(x)\big|\bigg|\int_{\{|y|\leq M\}} y\,\mu_t(\d y)\bigg|+\int_{\{|y|\leq M\}} \bigg(\int_0 ^1\big|\nabla^2 f(x+sy)\big|\,\d s\bigg) |y|^2\,\mu_t(\d y)\Bigg)\\
	&\qquad\leq 
    c_M\sup_{|y|>M}|f(x+y)-f(x)| 
    +C_M \max\bigg\{ \big|\nabla f(x)\big|, \sup_{|y|\leq M}\big|\nabla^2 f(x+y)\big|\bigg\}.
	\label{eq.apriori.gamma}
\end{align}
\end{enumerate}
\end{lemma}

\begin{proof}
    Let $M>0$.\ Then, for $t\geq h_0$,
    \[
     \frac{\mu_t\big(\big\{ y\in \R^d\,\big|\, |y|>M\big\}\big)}t+\frac1t\Bigg(\bigg| \int_{\{|y|\leq M\}}y\,\mu_t(\d y)\bigg|+\int_{\{|y|\leq M\}} |y|^2\,\mu_t(\d y)\Bigg)\leq \frac{(1\vee M)+M^2}{h_0}.
     \]
    Moreover, for $h\in (0,h_0)$,
	\[
	   \int_{\{|y|\leq M\}} |y|^2\,\mu_h(\d y) +\mu_h\big(\big\{ y\in \R^d\,\big|\, |y|>M\big\}\big)\leq \max\bigg\{M^2,\frac1{M^2}\bigg\} \int_{\R^d}1\wedge |y|^2\,\mu_h(\d y)<\infty
	\]
	and
	\[
	   \bigg| \int_{\{|y|\leq M\}}y\,\mu_h(\d y)\bigg|\leq \bigg| \int_{\{|y|\leq 1\}}y\,\mu_h(\d y)\bigg|+\max \bigg\{M,\frac1M\bigg\} \int_{\R^d}1\wedge |y|^2\,\mu_h(\d y)<\infty,
	\]
	which proves part a).\ Part b) is now a direct consequence of Taylor's theorem with integral form of the remainder.
\end{proof}	

\begin{remark}\label{rem.neccesity_T_1}
 Note that, by Lemma \ref{lem.apriori.gamma} a), Condition \eqref{cond.M} implies that \eqref{eq.M2.weaker} is satisfied for all $M>0$.\
 However, Condition \eqref{cond.M} does \textit{not} imply Condition \eqref{cond.T}.\ Indeed, let $d=1$ and
 \[
  \mu_t:=\tfrac{1-t}{2}\big(\de_t+\de_{-t}\big)+t\de_{t^{-1}}\quad \text{for all }t\in (0,1],
 \]
 where $\de_x$ denotes the Dirac measure on $\Bc(\R)$ with barycenter $x\in \R$.\ Moreover, we set $\mu_t:=\mu_1$ for all $t>1$. Then, 
 $\int_{\{|y|\leq 1\}} y\,\mu_h(\d y)=0$ and
 \[
  \frac1h\int_{\R^d} {1\wedge |y|^2}\,\mu_h(\d y)=1+\frac{1-h}{2h}\big(h^2+(-h)^2\big)=h(1-h)\leq \frac{1}{4}
 \]
 for all $h\in (0,1)$, so that Condition \eqref{cond.M} is satisfied with $h_0=1$.\ On the other hand,
 \[
 \frac{\mu_h\big(\big\{y\in \R^d\,\big|\, |y|>M\big\}\big)}h=1
 \]
 for all $M> 0$ and $h\in \big(0,M\wedge\frac1M\big)$, so that
 \[
 \lim_{h\downarrow 0} \frac{\mu_h\big(\big\{y\in \R^d\,\big|\, |y|>M\big\}\big)}h=1\quad\text{for all }M>0.
 \]
\end{remark}	

The following lemma provides a useful equivalent formulation of Condition \eqref{cond.M}. 

\begin{lemma}\label{lem.equiv.ass.M}
    Condition \eqref{cond.M} is equivalent to
 	\begin{equation}\label{eq.M1.prime}
 		\sup_{t>0}\bigg|\int_{\R^d} \frac{\ph(y)-\ph(0)}t\,\mu_t(\d y)\bigg|<\infty\quad \text{for all }\ph\in \Ccinf(\R^d).\tag{M'}
 	\end{equation}
\end{lemma}	

\begin{proof}[Proof]
	By Lemma \ref{lem.apriori.gamma} b), Condition \eqref{cond.M} implies \eqref{eq.M1.prime}.\ In order to prove the other implication, assume that \eqref{eq.M1.prime} is satisfied and let $\chi\in \Ccinf(\R^d)$ with $\chi(0)=1$, $0\leq \chi\leq 1$, and $\chi(x)=0$ for all $x\in \R^d$ with $|x|>1$, see, e.g., Lemma \ref{lem.smoothcutoff}.\ Then,
	\begin{align*}
		\sup_{t>0}\frac1t \int_{\R^d} &1\wedge |y|^2\,\mu_t(\d y)\leq \sup_{t>0}\frac1t \int_{\R^d} |y|^2\chi(y) +\big(1-\chi(y)\big)\,\mu_t(\d y)\\
		&= \sup_{t>0} \Bigg(\int_{\R^d} \frac{|y|^2\chi(y)}t\,\mu_t(\d y)+\bigg|\int_{\R^d} \frac{\chi(y)-\chi(0)}t\,\mu_t(\d y)\bigg|\Bigg)<\infty
	\end{align*}
    and
	\begin{align*}
		\sup_{t>0}\frac1t\bigg| \int_{\{|y|\leq 1\}} y\,\mu_t&(\d y)\bigg|\leq \sup_{t>0}\frac1t\Bigg(\bigg|\int_{\R^d} y\chi(y)\,\mu_t(\d y)\bigg|+\int_{\{|y|\leq 1\}} |y|\big(1-\chi(y)\big)\,\mu_t(\d y)\Bigg)\\
		&\quad\leq \sup_{t>0}\Bigg(\bigg|\int_{\R^d} \frac{y\chi(y)}{t}\,\mu_t(\d y)\bigg|+\bigg|\int_{\R^d} \frac{\chi(y)-\chi(0)}t\,\mu_t(\d y)\bigg|\Bigg)<\infty,
	\end{align*}
    where we used the fact that the maps $[x\mapsto|x|^2\chi(x)]\in \Ccinf(\R^d)$ and $[x\mapsto x\chi(x)]\in \Ccinf(\R^d)$ vanish at the origin. The proof is complete.
\end{proof}	

Condition \eqref{cond.M} entails a compactness condition for the family of measures $(\mu_t)_{t>0}$ as the following theorem shows. The proof adopts techniques from the proof of Courr\`ege's theorem, cf.\ \cite[Theorem 4.5.21]{jacob2001_volume_I}.

\begin{theorem}\label{prop.ex.levytriplet}
	Condition \eqref{cond.M} is satisfied if and only if every null sequence in $(0,\infty)$ has a subsequence $(h_n)_{n\in \N}$ such that
	\begin{equation}\label{eq.levygenerator}
		\lim_{n\to \infty} \int_{\R^d} \frac{f(x+y)-f(x)}{h_n}\,\mu_{h_n}=\big(L_{(b,\si,\nu)}f\big)(x) -cf(x) \quad \text{for all }f\in \Cc^2(\R^d)\text{ and }x\in \R^d,
	\end{equation}
    where $(b,\si,\nu)$ is a L\'evy triplet and $c\geq0$.\
    If Condition \eqref{cond.M} and Condition \eqref{cond.T} are satisfied, then $c=0$ in \eqref{eq.levygenerator}.
\end{theorem}	

\begin{proof}
 First assume that \eqref{cond.M} is not satisfied.\ Then, by Lemma \ref{lem.equiv.ass.M}, there exists a function $\ph\in \Ccinf(\R^d)$ and a sequence $(t_n)_{n\in \N}\subset (0,\infty)$ such that
 \[
  \bigg|\int_{\R^d} \frac{\ph(y)-\ph(0)}{t_n}\,\mu_{t_n}(\d y)\bigg|\geq n\quad \text{for all }n\in \N.
 \]
 Since $$\bigg|\int_{\R^d} \frac{\ph(y)-\ph(0)}t\,\mu_t(\d y)\bigg|\leq \frac{2}{h_0}\|\ph\|_\infty\quad \text{for all }t\geq h_0>0,$$ it follows that $t_n\to 0$ as $n\to \infty$ and, by construction, the sequence $\big(\int_{\R^d} \frac{\ph(y)-\ph(0)}{h_n}\,\mu_{h_n}(\d y)\big)_{n\in \N}$ does not converge in $\R$ for any subsequence $(h_n)_{n\in \N}$ of $(t_n)_{n\in \N}$.

 Now, assume that Condition \eqref{cond.M} is satisfied and let $(t_n)_{n \in \N}\subset (0,\infty)$ be a null sequence.\ Let $k\in \N$ and $f\in \Cb(\R^d)$. Then, by Lemma \ref{lem.apriori.gamma} a),
	\[
	\sup_{t>0} \bigg|\int_{\{|y|\leq k\}} \frac{|y|^2}h f(y)\,\mu_t(\d y)\bigg|\leq \|f\|_\infty \sup_{h\in (0,h_0)}\int_{\{|y|\leq k\}} \frac{|y|^2}t \,\mu_t(\d y)\leq C_k\|f\|_\infty.
	\]
	Using Prokhorov's theorem and a diagonal argument, there exists a subsequence $(h_n)_{n \in \N}$ of $(t_n)_{n \in \N}$ and a sequence $(\eta_k)_{k\in \N}$ of finite Borel measures with
	$$\eta_k\big(\big\{y\in \R^d\,\big|\, |y|>k\big\}\big)=0\quad \text{for all }k\in \N$$
	and
	\begin{equation}\label{eq.convergence.eta_k}
	\int_{\R^d} f(y)\,\eta_k(\d y)=\lim_{n\to \infty} \int_{\{|y|\leq k\}}\frac{|y|^2}{h_n}f(y)\,\mu_{h_n}(\d y)\quad\text{for all }k\in \N \text{ and }f\in \Cb(\R^d).
	\end{equation}
	In particular, $\eta_k (B)\leq \eta_{k+1}(B)$ for all $k\in \N$ and $B\in \Bc(\R^d)$, so that $\eta(B):=\sup_{k\in \N}\eta_k(B)$, for all $B\in \Bc(\R^d)$, defines a locally finite measure $\eta\colon \Bc(\R^d)\to [0,\infty]$ with
	\[
	\int_{\R^d} f(y)\,\eta(\d y)=\lim_{n\to \infty}\int_{\R^d} \frac{|y|^2}{h_n} f(y)\,\mu_{h_n}(\d y)\in \R\quad\text{for all }f\in \Cc(\R^d),
	\]
    and 
	\begin{equation}\label{eq.monapprox.eta}
		\int_{\R^d} f(y)\,\eta(\d y)=\sup_{k\in \N}  \int_{\R^d} f(y)\,\eta_k(\d y)\quad\text{for all }f\in \Cb(\R^d)\text{ with }f\geq0.
	\end{equation}
    We define
    \[
	\nu(B):=\int_{B\setminus\{0\}} \frac{1}{|y|^2} \,\eta(\d y)\quad \text{for all }B\in \Bc(\R^d).
	\]
	Now, let $\chi\in \Ccinf(\R^d)$ with $\chi(x)=1$ for all $x\in \R^d$ with $|x|\leq 1$.\ After potentially passing to a subsequence, by Lemma \ref{lem.equiv.ass.M}, we may w.l.o.g.\ assume that
	\begin{align*} 
		b_0:=\lim_{n\to \infty}\frac1{h_n}\int_{\R^d} y \chi(y)\, \mu_{h_n}(\d y)\in \R^d\quad\text{and}\quad \si_0:=\lim_{n\to \infty}\frac1{h_n}\int_{\R^d} yy^T \chi(y)\, \mu_{h_n}(\d y)\in \R^{d\times d}
	\end{align*}
	exist. By Lemma \ref{lem.apriori.gamma} a), after potentially passing to yet another subsequence, using a diagonal argument, we may w.l.o.g.\ assume that
    \begin{equation}\label{eq.bound.c_k}
     \lim_{n\to \infty} \frac{\mu_{h_n}\big(\big\{y\in \R^d\,\big|\, |y|>k\big\}\big)}{h_n}\in [0,c_k]
    \end{equation}
    exists, and we define
    \begin{equation}\label{eq.def.c}
    c:=\inf_{k\in \N} \lim_{n\to \infty} \frac{\mu_{h_n}\big(\big\{y\in \R^d\,\big|\, |y|>k\big\}\big)}{h_n}=\lim_{k\to\infty}\lim_{n\to \infty} \frac{\mu_{h_n}\big(\big\{y\in \R^d\,\big|\, |y|>k\big\}\big)}{h_n}.
    \end{equation}
	Using first \eqref{eq.monapprox.eta} and \eqref{eq.def.c}, and then \eqref{eq.convergence.eta_k}, for all $\ep>0$, there exist $k\in \N$ and $n_0\in \N$ such that, for all $n\geq n_0$,
	\begin{align*}
		\bigg|\int_{\R^d} \frac{1- \chi(y)}{h_n}\, \mu_{h_n}(\d y)-\int_{\R^d\setminus \{0\}} &1- \chi(y)\, \nu(\d y)-c\bigg|< \frac\ep2\\
		&+\bigg|\int_{\{|y|\leq k\}} \frac{1- \chi(y)}{h_n}\, \mu_{h_n}(\d y)-\int_{\{|y|\leq k\}} \frac{1- \chi(y)}{|y|^2}\, \eta_k(\d y)\bigg|<\ep,
	\end{align*}
	which shows that
	\[
	\lim_{n\to \infty}\int_{\R^d} \frac{1- \chi(y)}{h_n}\, \mu_{h_n}(\d y)=\int_{\R^d\setminus \{0\}} 1- \chi(y)\, \nu(\d y)+c.
	\]
	Moreover, by construction,
	\[
	\int_{\{|y|\leq k\}} |y|^2\,\nu(\d y)=\eta_k(\R^d)=\lim_{n\to \infty}  \int_{\{|y|\leq k\}} \frac{|y|^2}{h_n}\,\mu_{h_n}(\d y)\in \R\quad \text{for all }k\in \N.
	\]
	We have therefore shown that
	\[
	\nu\big(\{0\}\big)=0\quad\text{and}\quad \int_{\R^d\setminus\{0\}} 1\wedge |y|^2\,\nu(\d y)<\infty.
	\]
	Now, let $f\in \Cc^2(\R^d)$, and define $(Tf)(0):=0$ and
	\[
	(Tf)(y):=\frac{1}{|y|^2}\bigg(f(y)-f(0)\chi(y)-\langle \nabla f(0),y\rangle \chi(y)-\frac12 y^T\nabla^2 f(0)y \chi(y)\bigg)
	\]
	for all $y\in \R^d\setminus\{0\}$.\ Then, by Taylor's theorem, $Tf\in \Cc(\R^d)$ and it follows that
	\begin{align*}
		\lim_{n\to \infty}&\int_{\R^d}\frac{f(y)-f(0)}{h_n}\,\mu_{h_n}(\d y)=-f(0)\bigg(\int_{\R^d} 1-\chi(y)\,\nu(\d y)+c\bigg)+\langle b_0,\nabla f(0)\rangle\\
        &+\frac12\trace\big(\sigma_0 \nabla^2f(0)\big)+\int_{\R^d\setminus\{0\}} f(y)-f(0)\chi(y)-\langle \nabla f(0),y\rangle \chi(y)-y^T\nabla^2 f(0)y \chi(y)\,\nu(\d y).
	\end{align*}
	Define
	\[
	b_1:= \int_{\R^d\setminus\{0\}} y \big(\chi(y)-\eins_{[0,1]}(|y|)\big)\,\nu(\d y)\in \R^{d\times d}\quad\text{and}\quad\sigma_1:=\int_{\R^d\setminus\{0\}} yy^T \chi(y)\,\nu(\d y)\in \R^{d\times d}.
	\]
	Then, $\si_0$ and $\si_1$ are symmetric and, by \eqref{eq.monapprox.eta},
	\[
	\xi^T\si_1\xi=\int_{\R^d\setminus\{0\}} |\langle y,\xi\rangle|^2 \chi(y)\,\nu(\d y)\leq  \lim_{n\to \infty}\int_{\R^d} \frac{|\langle y,\xi\rangle|^2}{h_n} \chi(y)\,\mu_{h_n}(\d y)=  \xi^T\si_0\xi\quad\text{for all }\xi\in \R^d.
	\]
	Therefore, $\si:=\si_0-\si_1\in \R^{d\times d}$ is symmetric and positive semidefinite.\ Moreover, we define $b:=b_0-b_1\in \R^d$. Then, we have shown that $(b,\si,\nu)$ is a L\'evy triplet and
	\[
	\lim_{n\to \infty}\int_{\R^d}\frac{f(x+y)-f(x)}{h_n}\,\mu_{h_n}(\d y)=(L_{(b,\si,\nu)}f)(x)-cf(x)\quad\text{for all }f\in \Cc^2(\R^d)\text{ and }x\in \R^d.
	\]
    By \eqref{eq.bound.c_k}, $c\geq0$ and $c=0$ if Condition \eqref{cond.T} is additionally satisfied, i.e., $\inf_{M>0} c_M=0$.
\end{proof}

We conclude this section with an priori estimate for the difference quotient.

\begin{lemma}\label{lem.diff.quotient.Lp}
 Assume that Condition \eqref{cond.M} is satisfied and let $p\in [1,\infty)$, $f\in W_p^2(\R^d)$, and $M>0$.\ Then, 
 \begin{equation}\label{eq.apriori.Lp}
\sup_{t>0}\bigg\|\int_{\R^d}\frac{f(\,\cdot+y)-f}{t}\,\mu_t(\d y)\bigg\|_p\leq 2c_M\|f\|_p+C_M\max\big\{\big\|\nabla f\big\|_p,\big\|\nabla^2f\big\|_p\big\}.
 \end{equation}
 \end{lemma}

 \begin{proof}
 First, let $f\in \Cc^2(\R^d)$. Lemma \ref{lem.apriori.gamma}, the triangle inequality, Lemma \ref{lem.convolution}, and Tonelli's theorem imply that, for all $t>0$,
	\begin{align*}
		\bigg(\int_{\R^d} \bigg|\int_{\R^d}\frac{f(x+y)-f(x)}{t}\,&\mu_t(\d y)\bigg|^p\,\d x\bigg)^{1/p}\\
        &\leq \bigg(\int_{\R^d} \bigg|\int_{\{|x|>M\}} \frac{f(x+y)-f(x)}t\, \mu_t(\d y)\bigg|^p\, \d x\bigg)^{1/p}\\
		&\quad + \bigg( \int_{\R^d}\big|\nabla f(x)\big|^p\bigg|\int_{\{|y|\leq M\}} \frac{y}t\,\mu_t(\d y)\bigg|^p\,\d x\bigg)^{1/p}\\
		&\quad +\bigg(\int_0^1\int_{\R^d}\bigg|\int_{\{|y|\leq M\}} \big|\nabla^2 f(x+sy)\big| \frac{|y|^2}t\,\mu_t (\d y)\bigg|^p\,\d x \,\d s\bigg)^{1/p}\\
		&\leq 2\|f\|_p\frac{\mu_h\big(\big\{ y\in \R^d\,\big|\, |y|>M\big\}\big)}t+\big\|\nabla f\big\|_p \bigg|\int_{\{|y|\leq M\}} \frac{y}{t}\,\mu_t(\d y)\bigg|\\
		&\quad + \big\|\nabla^2f\big\|_p \int_{\{|y|\leq M\}}\frac{|y|^2}t\,\mu_t (\d y)\\
		&\leq 2c_M\|f\|_p+C_M \max\big\{\big\|\nabla f\big\|_p,\big\|\nabla^2f\big\|_p\big\}.
	\end{align*}
    Now, the statement for $f\in W_p^2(\R^d)$ follows from the density of $\Ccinf(\R^d)$ in $W_p^2(\R^d)$, see, e.g., \cite[Corollary 3.23]{MR2424078}.
 \end{proof}

\section{Discussion of Condition (T)}\label{sec.T}

The aim of this section is to strengthen the convergence in \eqref{eq.levygenerator} in various directions.\ For this, it is unavoidable to assume Condition \eqref{cond.T}, as we will see in various examples below.\
We start with the following characterization of Condition \eqref{cond.T}.

\begin{lemma}\label{lem.equiv.ass.T}
 Assume that Condition \eqref{cond.M} is satisfied. Then, for all $\ep>0$, there exists $M_\ep>0$ such that \eqref{eq.Ass.M2} is satisfied if and only if  there exists a function $\chi_\ep\in \Ccinf(\R^d)$ with $\chi_\ep(0)=1$, $0\leq \chi_\ep\leq 1$, and 
 	\begin{equation}\label{eq.M2.prime}
 		\limsup_{h\downarrow 0}\bigg(-\int_{\R^d} \frac{\chi_\ep(y)-\chi_\ep(0)}h\,\mu_h(\d y)\bigg)< \ep.\tag{T'}
 	\end{equation}
\end{lemma}	

	\begin{proof}[Proof]
	Let $\ep>0$ and $\chi_\ep\in \Ccinf(\R^d)$ with $\chi_\ep(0)=1$ and $0\leq \chi_\ep\leq 1$ such that \eqref{eq.M2.prime} is satisfied. Then, there exists $M_\ep> 0$ such that $\chi_\ep(x)=0$ for all $x\in \R^d$ with $|x|>M_\ep$, which implies that
	\[
	\limsup_{h\downarrow0}\frac{\mu_h\big(\big\{y\in \R^d\,\big|\, |y|>M\big\}\big)}h\leq \limsup_{h\downarrow 0}\bigg(\int_{\R^d} \frac{1-\chi_\ep(y)}h\,\mu_h(\d y)\bigg)< \ep.
	\]
    In order to prove the other implication, let $M_\ep>0$ such that \eqref{eq.Ass.M2}.\ Then, there exists a function $\chi_\ep\in \Ccinf(\R^d)$ with $\chi_\ep(x)=1$ for all $x\in \R^d$ with $|x|\leq M_\ep$ and $0\leq \chi_\ep\leq 1$, see, e.g., Lemma \ref{lem.smoothcutoff}. By Lemma \ref{lem.apriori.gamma} b) with $M=M_\ep$, $f=\chi_\ep$, and $x=0$, we find that
    \begin{align*}
     \limsup_{h\downarrow 0} \bigg(-\int_{\R^d}\frac{\chi_\ep(y)-\chi_\ep(0)}{h}\,\mu_h(\d y)\bigg)&=\limsup_{h\downarrow 0}\bigg|\int_{\R^d}\frac{\chi_\ep(y)-\chi_\ep(0)}{h}\,\mu_h(\d y)\bigg|\\
     &\leq \limsup_{h\downarrow 0}\frac{\mu_h\big(\big\{y\in \R^d\,\big|\,|y|>M_\ep\big\}\big)}h<\ep.
    \end{align*}
    The proof is complete.
\end{proof}

In a first step, we focus on extending the convergence in \eqref{eq.levygenerator} to the space $\Cb^2(\R^d)$.\ For this, it is unavoidable to assume Condition \eqref{cond.T} as the following remark indicates.

\begin{remark}\label{rem.neccesity_T_2}
We remark that, in order to obtain pointwise convergence of the difference quotients on $\Cb^\infty(\R^d)$ instead of $\Cc^2(\R^d)$, see Proposition \ref{prop.ex.levytriplet}, it is essential to assume Condition \eqref{cond.T}.\
To that end, we revisit the example given in Remark \ref{rem.neccesity_T_1}, i.e., we consider the case $d=1$, 
\[
 \mu_t:=\frac{1-t}{2}\big(\de_t+\de_{-t}\big)+t\de_{t^{-1}}\quad \text{for }t\in (0,1],
\]
$\mu_t:=\mu_1$ for $t>1$, and $h_0:=1$.\ Recall that the family $(\mu_t)_{t>0}$ satisfies Condition \eqref{cond.M} but \textit{not} Condition \eqref{cond.T}.\
Let $(h_n)_{n\in\N}\subset (0,1)$ with $h_n\to0$ as $n\to\infty$.\ Then, for any $f\in \Cb(\R)$ such that the sequence $\big(f(h_n^{-1})\big)_{n\in \N}\subset \R$ does not converge, it follows that
\[
\int_{\R}\frac{f(y)-f(0)}{h}\,\mu_h(\d y )=f\bigg(\frac1{h_n}\bigg)-f(0)\quad \text{for all }n\in \N,
\]
so that the sequence $\big(\int_{\R}\frac{f(y)-f(0)}{h_n}\,\mu_{h_n}(\d y )\big)_{n\in \N}$ does not converge.\

Moreover, the set of all $f\in \Cbinf(\R)$ for which the sequence $\big(f(h_n^{-1})\big)_{n\in \N}$ does not converge is sequentially dense in $\Cb(\R)$, endowed with the mixed topology. Indeed, let $f\in \Cb(\R)$ and $(a_n)_{n\in \N}\subset \R$ be a bounded sequence that does not converge and satisfies $a_n=a_m$ for all $m,n\in \N$ with $h_n=h_m$.\ Then, using a suitable mollification, for all $k\in \N$, there exists $f_k\in \Cbinf(\R)$ such that $$\sup_{|x|\leq k}|f(x)-f_k(x)|\leq \frac1k,$$ $f_k(h_n^{-1})=a_n$ for all $n\in \N$ with $h_n<\frac1k$, and $\sup_{k\in \N}\|f_k\|_\infty<\infty$.\footnote{Note that, for all $m\in \N$, the set of all $n\in \N$ such that $|h_n-h_m|<h_m$ is finite, since $h_n\to 0$ as $n\to \infty$.\ Hence, for all $m\in \N$, there exists some $\de>0$ such that $|h_n^{-1}-h_m^{-1}|\geq\de$ for all $n\in \N$ with $h_n\neq h_m$.}\ By construction, $f_k\to f$ in the mixed topology as $k\to \infty$ and the sequence $\big(f_k(h_n^{-1})\big)_{n\in \N}$ does not converge for any $k\in \N$.
\end{remark}

We have the following equivalent description of the conditions \eqref{cond.M} and \eqref{cond.T}.

\begin{proposition}\label{prop.conv.levy.mixed}
The following statements are equivalent.
\begin{enumerate}
    \item[(i)] The conditions \eqref{cond.M} and \eqref{cond.T} are satisfied.
    \item[(ii)] For every null sequence in $(0,\infty)$, there exists a subsequence $(h_n)_{n\in \N}$ and a L\'evy triplet $(b,\si,\nu)$ such that, for all $f\in \Cb^2(\R^d)$,
		\begin{equation}\label{eq.gen.levy.mixed}
		\lim_{n\to \infty} \int_{\R^d}\frac{f(\,\cdot+y)-f}{h_n}\,\mu_{h_n}(\d y) =L_{(b,\sigma,\nu)}f\quad \text{in the mixed topology.}
		\end{equation}
        \end{enumerate}
\end{proposition}

\begin{proof}
  We start with the implication (i)$\,\Rightarrow\,$(ii).\ To that end, let $f\in \Cb^2(\R^d)$ and $f_k:=f\chi_k$ for all $k\in \N$, where $(\chi_k)_{k\in \N}\subset \Ccinf(\R^d)$ with $\chi_k(y)=1$ for all $y\in \R^d$ with $|y|\leq k$ and $0\leq \chi_k\leq \chi_{k+1}\leq 1$ for all $k\in \N$.\ Then, by \eqref{eq.apriori.gamma}, for all $x\in \R^d$, $\ep>0$, and $k\in \N$ with $k\geq |x|+M_\ep$,
    \begin{align*}
        \sup_{h\in (0,h_0)}&\bigg|\int_{\R^d}\frac{f(x+y)-f(x)}{h}\mu_h(\d y)-\int_{\R^d}\frac{f_k(x)-f_k(x)}{h}\,\mu_h(\d y)\bigg|\\
        & \qquad\qquad\qquad\qquad\quad\quad  =  \sup_{h\in (0,h_0)}\bigg|\frac{\big(f(1-\chi_k)\big)(x+y)-\big(f(1-\chi_k)\big)(x)}{h}\bigg| \leq 2\|f\|_{\infty} \ep.
    \end{align*}
    Since, by the dominated convergence theorem, $\big(L_{(b,\si,\nu)}f_k\big)(x)\to \big(L_{(b,\si,\nu)}f\big)(x)$ as $k\to \infty$ for all $x\in \R^d$, we thus find that
    \[
   \lim_{n\to\infty}  \int_{\R^d}\frac{f(x+y)-f(x)}{h_n}\mu_{h_n}(\d y)=\big(L_{(b,\si,\nu)}\big)f(x)\quad \text{for all }x\in \R^d.
    \]
Now, by \eqref{eq.apriori.gamma}, for all $\ep>0$ and $x_1,x_2\in \R^d$,
	\begin{align}
    \notag \sup_{h\in(0,h_0)} &\bigg|\int_{\R^d}\frac{f(x_1+y)-f(x_1)}{h}\,\mu_h(\d y)-\int_{\R^d}\frac{f(x_2+y)-f(x_2)}{h}\,\mu_h(\d y)\bigg|\\
		&\quad \leq \sup_{h\in(0,h_0)} \bigg|\int_{\R^d}\frac{\big(f(x_1+y)-f(x_2+y)\big)-\big(f(x_1)-f(x_2)\big)}{h}\,\mu_h(\d y)\bigg| \notag\\
		&\quad \leq 2\ep \|f\|_\infty+C_{M_\ep}\max\bigg\{\big|\nabla f(x_1)-\nabla f(x_2)\big|,\sup_{|y|\leq M_\ep}\big|\nabla^2 f(x_1+y)-\nabla f(x_2+y)\big|\bigg\}, \label{eq.estimate.difference}
	\end{align}
	 so that, by compactness of the set $\{x\in \R^d\,|\, |x|\leq r\}$, for all $r\geq 0$,
	\[
	\sup_{|x|\leq r}\bigg|\int_{\R^d}\frac{f(x+y)-f(x)}{h_n}\,\mu_{h_n}(\d y)-\big(L_{(b,\si,\nu)}f\big)(x)\bigg|\to 0\quad\text{as }n\to \infty,
	\]
    since $L_{(b,\si,\nu)}f$ is continuous by the dominated convergence theorem. Moreover, by \eqref{eq.apriori.gamma}, 
    \[
    \sup_{h\in (0,h_0)}\bigg\| \int_{\R^d}\frac{f(\,\cdot+y)-f}{h}\,\mu_h(\d y)\bigg\|_\infty<\infty,
    \]
    so that $\int_{\R^d}\frac{f(\,\cdot\, +y)-f}{h_n}\,\mu_{h_n}(\d y) \to L_{(b,\sigma,\nu)}f$ in the mixed topology as $n\to\infty$.

    Now, assume that (ii) is satisfied.\ Then, by Theorem \ref{prop.ex.levytriplet}, Condition \eqref{cond.M} is satisfied.\ Assume towards a contradiction that Condition \eqref{cond.T} is not satisfied, i.e., there exists $\ep>0$ such that, for all $n\in \N$, there exists $t_n\in \big(0,\frac{1}n\big)$ with
    \[
     \frac{\mu_{t_n}\big(\big\{y\in \R^d\,\big|\, |y|>n \big\}\big)}{t_n}\geq \ep.
    \]
    Then, $(t_n)_{n\in \N}\subset (0,\infty)$ is a null sequence and, for all $M>0$ and $n\in \N$ with $n\geq M$,
    \[
     \frac{\mu_{t_n}\big(\big\{y\in \R^d\,\big|\, |y|>M \big\}\big)}{t_n}\geq \frac{\mu_{t_n}\big(\big\{y\in \R^d\,\big|\, |y|>n \big\}\big)}{t_n}\geq \ep,
    \]
so that
\begin{equation}\label{eq.contradiction.tight}
\liminf_{n\to \infty} \frac{\mu_{t_n}\big(\big\{y\in \R^d\,\big|\, |y|>M \big\}\big)}{t_n}\geq \ep\quad \text{for all }M>0.
\end{equation}
    By (ii), there exist a subsequence $(h_n)_{n\in \N}$ of $(t_n)_{n\in \N}$ and a L\'evy triplet $(b,\si,\nu)$ such that \eqref{eq.gen.levy.mixed} holds for all $f\in \Cb^2(\R^d)$. Since $\nu$ is a L\'evy measure, there exists a function $\chi\in \Ccinf(\R^d)$ with $\chi(x)=1$ for all $x\in \R^d$ with $|x|\leq 1$ and $0\leq \chi\leq 1$ such that
    \[
     \int_{\R^d\setminus \{0\}} 1-\chi(y)\,\nu(\d y)< \ep.
    \]
     Then, $\nabla \chi(0)=0$, $\nabla^2 \chi(0)=0$, and there exists a constant $M>0$ such that $\chi(x)=0$ for all $x\in \R^d$ with $|x|>M$. Hence,
    \begin{align*}
     \limsup_{n\to \infty} \frac{\mu_{h_n}\big(\big\{y\in \R^d\,\big|\, |y|>M \big\}\big)}{h_n}&\leq \lim_{n\to \infty} \int_{\R^d} \frac{1-\chi(y)}{h_n}\,\mu_{h_n}(\d y)=-\big(L_{(b,\si,\nu)}\chi\big)(0)\\
     &\leq \int_{\R^d\setminus \{0\}} 1-\chi(y)\,\nu(\d y)< \ep,
    \end{align*}
    which is a contradiction to \eqref{eq.contradiction.tight}.\ The proof is complete.        
\end{proof}

Assuming conditions \eqref{cond.M} and \eqref{cond.T}, the convergence in \eqref{eq.levygenerator} can be improved also in other directions as the following proposition shows.\ Here and throughout, we define $L_{(b,\sigma,\nu)}f$, for $f\in W_p^2(\R^d)$ and $p\in [1,\infty]$, via \eqref{eq.def.levygenerator} using weak derivatives instead of classical ones. Then, by Lemma \ref{lem.convolution} together with the identity $$\int_{\{y\leq 1\}}f(x+y)-f(x)-\nabla f(x)y\, \nu(\d y)=\frac12\int_0^1\int_{\{y\leq 1\}}y^T\nabla^2f(x+sy)y\, \nu(\d y)\,\d s,$$ it follows that $L_{(b,\sigma,\nu)}f\in L_p(\R^d)$ for $f\in W_p^2(\R^d)$ and $p\in [1,\infty]$.

\begin{proposition}\label{prop.levygenerator}
	Assume that the conditions \eqref{cond.M} and \eqref{cond.T} are satisfied.\ Let $(h_n)_{n\in \N}\subset (0,\infty)$ be a null sequence and $(b,\si,\nu)$ be a L\'evy triplet such that \eqref{eq.levygenerator} holds with $c=0$.
	\begin{enumerate}
    \item[a)] Let $f\in \BUC^2(\R^d)$.\ Then, $L_{(b,\sigma,\nu)}f\in \BUC(\R^d)$ and
		\begin{equation}\label{eq.conv.levy.unif}
		\lim_{n\to \infty}\bigg\|\int_{\R^d}\frac{f(\,\cdot+y)-f}{h_n}\,\mu_{h_n}(\d y)-L_{(b,\sigma,\nu)}f\bigg\|_\infty=0.
		\end{equation}
        If $f\in \Cnt_0^2(\R^d)$, then $L_{(b,\sigma,\nu)}f\in \Cnt_0(\R^d)$.
		\item[b)] Let  $p\in [1,\infty)$ and $f\in W_p^2(\R^d)$.\ Then, $L_{(b,\sigma,\nu)}f\in L_p(\R^d)$ and
		\begin{equation}\label{eq.conv.levy.Lp}
		\lim_{n\to \infty}\bigg\|\int_{\R^d}\frac{f(\,\cdot+y)-f}{h_n}\,\mu_{h_n}(\d y)-  L_{(b,\sigma,\nu)}f\bigg\|_p=0.
		\end{equation}
	\end{enumerate}	
\end{proposition}	

\begin{proof} 
We first show \eqref{eq.conv.levy.unif} and \eqref{eq.conv.levy.Lp} for $f\in \Cc^2(\R^d)$. To that end, let $f\in \Cc^2(\R^d)$.\ Then, by dominated convergence, $L_{(b,\si,\nu)}f\in \Cnt_0(\R^d)$.\ Moreover, for all $\ep>0$, there exists some $r_\ep\geq0$ such that $f(x+y)=0$ for all $x,y\in \R^d$ with $|x|>r_\ep$ and $|y|\leq M_\ep$. Hence, for all $\ep>0$,
	\begin{align*}
	\sup_{|x|>r_\ep} \bigg|\int_{\R^d}\frac{f(x+y)-f(x)}{h}\,\mu_h(\d y )\bigg|&= \sup_{|x|>r_\ep} \bigg|\int_{\{|y|>M_\ep\}}\frac{f(x+y)}{h}\,\mu_h(\d y )\bigg|\\
	&\leq \|f\|_\infty\frac{\mu_h\big(\big\{ y\in \R^d\,\big|\, |y|>M_\ep\big\}\big)}h\leq\ep
	\end{align*}
    for $h\in (0,h_0)$ sufficiently small.\
 Therefore, by Proposition \ref{prop.conv.levy.mixed}, it follows that \ref{eq.conv.levy.unif} holds for $f\in \Cc^2(\R^d)$.
 
 Now, let $p\in [1,\infty)$ and $f\in \Cc^2(\R^d)$.\ Then, by Theorem \ref{prop.ex.levytriplet}, Lemma \ref{lem.diff.quotient.Lp}, and Fatou's lemma, $L_{(b,\si,\nu)}f\in L_p(\R^d)$.\ Moreover, by Lemma \ref{lem.convolution},
\begin{align*}
	 \bigg(\int_{\{|x|>r_\ep\}} \bigg|\int_{\R^d}\frac{f(x+y)-f(x)}{h}\,&\mu_h(\d y )\bigg|^p\,\d x\bigg)^{1/p}\\
     &=  \bigg(\int_{\{|x|>r_\ep\}} \bigg|\int_{\{|y|>M_\ep\}}\frac{f(x+y)}{h}\,\mu_h(\d y )\bigg|^p\,\d x\bigg)^{1/p}\\
     &\leq  \bigg(\int_{\R^d} \bigg|\int_{\{|y|>M_\ep\}}\frac{f(x+y)}{h}\,\mu_h(\d y )\bigg|^p\,\d x\bigg)^{1/p}\\
	 &\leq \|f\|_p\frac{\mu_h\big(\big\{ y\in \R^d\,\big|\, |y|>M_\ep\big\}\big)}h\leq\ep.
	\end{align*}
    Hence, by Proposition \ref{prop.conv.levy.mixed}, \eqref{eq.conv.levy.Lp} follows for all $p\in [1,\infty)$ and $f\in \Cc^2(\R^d)$, i.e., we have shown that
 \begin{equation}\label{eq.conv.Cc2}    
 \lim_{n\to \infty}\bigg\|\int_{\R^d}\frac{f(\,\cdot+y)-f}{h_n}\,\mu_{h_n}(\d y)-  L_{(b,\sigma,\nu)}f\bigg\|_p=0\quad \text{for all }f\in \Cc^2(\R^d)\text{ and }p\in [1,\infty].
 \end{equation}
 Next, observe that, by Fubini's theorem,
	\begin{equation}\label{eq.identity.convolution}
	 \int_{\R^d}\frac{(f\ast \eta)(\,\cdot+y)-(f\ast \eta)}h\,\mu_h(\d y)=f\ast \bigg(\int_{\R^d}\frac{\eta(\,\cdot+ y)- \eta}{h}\,\mu_h(\d y)\bigg) 
	\end{equation}
    for all $h\in (0,h_0)$, $f\in L_p(\R^d)$ with $p\in [1,\infty]$ and $\eta\in \Cc^2(\R^d)$.\
 Combining \eqref{eq.conv.Cc2} and \eqref{eq.identity.convolution} with Lemma \ref{lem.convolution}, it follows that
\begin{equation}\label{eq.generator}
 L_{(b,\sigma,\nu)}(f\ast\eta)= f\ast \big(L_{(b,\sigma,\nu)}\eta\big) \quad\text{for all }f\in L_p(\R^d)\text{ with }p\in [1,\infty]\text{ and }\eta\in \Cc^2(\R^d).
\end{equation}
Since $L_{(b,\sigma,\nu)}f\in L_p(\R^d)$ for $f\in W_p^2(\R^d)$ and $p\in [1,\infty]$,
using the definition of the weak derivative, it follows that
\begin{equation}\label{eq.identity.convolution.main}
L_{(b,\sigma,\nu)}(f\ast\eta)= f\ast \big(L_{(b,\sigma,\nu)}\eta\big)=\big(L_{(b,\sigma,\nu)}f\big)\ast\eta\quad \text{for all }f\in W_p^2(\R^d) \text{ with }p\in [1,\infty].
\end{equation}
Now, using Friedrich's mollifier, \eqref{eq.conv.levy.unif} and \eqref{eq.conv.levy.Lp} for $f\in \BUC^2(\R^d)$ and $f\in W_p^2(\R^d)$ with $p\in [1,\infty)$ follow from  Lemma \ref{lem.apriori.gamma} and Lemma \ref{lem.diff.quotient.Lp} together with the fact that $L_{(b,\sigma,\nu)}f$ is an element of $\BUC(\R^d)$ and $L_p(\R^d)$ for $f\in \BUC^2(\R^d)$ and $f\in W_p^2(\R^d)$ with $p\in [1,\infty)$, respectively. Since $\Cnt_0(\R^d)$ is a Banach space, it follows that $L_{(b,\sigma,\nu)}f\in \Cnt_0(\R^d)$ for $f\in \Cnt_0^2(\R^d)$ and the proof is complete.
\end{proof}

\begin{remark}\label{rem.necessity_T_3}
We point out that Condition \eqref{cond.M} alone is again not enough to ensure the convergence results in the previous proposition, see Remark \ref{rem.neccesity_T_2}.\ Indeed, despite the fact that, under Condition \eqref{cond.M}, the difference quotients converge pointwise for all $f\in \Cc^2(\R^d)$, and are uniformly bounded w.r.t.\ the supremum norm or $L_p$-norm, it is possible that they do not converge w.r.t.\ the supremum norm or $L_p$-norm for \textit{any} nonzero element of $\Cnt_0(\R^d)$ or $L_p(\R^d)$ for $p\in [1,\infty)$, respectively.\ 
To that end, we revisit the example given in Remark \ref{rem.neccesity_T_1} yet another time.\ Let $d=1$, 
\[
 \mu_t:=\frac{1-t}{2}\big(\de_t+\de_{-t}\big)+t\de_{t^{-1}}\quad \text{for }t\in (0,1],
\]
$\mu_t:=\mu_1$ for $t>1$, and $h_0:=1$.

First, let $f\in\Cnt_0(\R)\setminus \{0\}$.\ Then, there exists some $r\geq\frac12$ such that $\sup_{|x|>r}|f(x)|\leq \frac{\|f\|_\infty}{2}$.\ Now, if $h\in \big(0,\frac1{2r}\big)$, the inverse triangle inequality implies that 
\begin{equation}\label{eq.ex.inverse.triangle}
\bigg|z-\frac1h\bigg|\geq \frac1h-|z|> r \quad \text{for all }z\in \R\text{ with }|z|\leq r.
\end{equation}
Hence, for all $h\in \big(0,\frac1{2r}\big)$, 
\[
\sup_{|x|>r} \bigg|\int_{\R}\frac{f(x+y)-f(x)}{h}\,\mu_h(\d y )\bigg|= \sup_{|x|>r} \bigg|f\bigg(x+\frac1h\bigg)-f(x)\bigg|\geq\sup_{|z|\leq r} |f(z)|-\frac{\|f\|_\infty}{2}=\frac{\|f\|_\infty}{2}.
\]
Since $\int_{\R^d}\frac{f(\,\cdot\, +y)-f(x)}{t}\,\mu_t(\d y )\in \Cnt_0(\R)$ for all $t>0$ and $\Cnt_0(\R)$, endowed with the supremum norm, is a Banach space, the sequence $\big(\int_{\R}\frac{f(\,\cdot\,+y)-f(x)}{h_n}\,\mu_{h_n}(\d y )\big)_{n\in \N}$ does not converge w.r.t.\ the supremum norm for any sequence $(h_n)_{n\in\N}\subset (0,1)$ with $h_n\to0$ as $n\to\infty$.

Now, let $p\in [1,\infty)$ and $f\in L_p(\R)\setminus\{0\}$.\ Then, there exists some $r\geq \frac12$ such that $\big(\int_{\{|x|> r\}} |f(x)|^p\,\d x\big)^{1/p}\leq \frac{\|f\|_p}{4}$, which implies that
\[
\|f\|_p\leq \bigg(\int_{\{|x|\leq  r\}} |f(x)|^p\,\d x\bigg)^{1/p}+\frac{\|f\|_p}{4}
\]
and, consequently, $\big(\int_{\{|x|\leq  r\}} |f(x)|^p\,\d x\big)^{1/p}\geq \frac{3\|f\|_p}4.$
Using again \eqref{eq.ex.inverse.triangle} together with the substitution $z=x+\frac1h$, we thus obtain,
for all $h\in \big(0,\frac1{2r}\big)$,
\begin{align*}
 \bigg(\int_{\{|x|>r\}} \bigg|\int_{\R^d}\frac{f(x+y)-f(x)}{h}\,\mu_h(\d y )\bigg|^p\,\d x\bigg)^{1/p}&= \bigg(\int_{\{|x|>r\}} \bigg|f\bigg(x+\frac1h\bigg)-f(x)\bigg|^p\,\d x\bigg)^{1/p}\\
 &\geq \bigg(\int_{\{|z|\leq r\}} |f(z)|\,\d z\bigg)^{1/p}-\frac{\|f\|_p}4\geq \frac{\|f\|_p}2.
\end{align*}
Hence, by \cite[Corollary 4.5.5]{MR2267655}, the sequence $\big(\int_{\R^d}\frac{f(\,\cdot\,+y)-f(x)}{h_n}\,\mu_{h_n}(\d y )\big)_{n\in \N}$ does not converge in $L_p(\R)$ for any sequence $(h_n)_{n\in\N}\subset (0,1)$ with $h_n\to0$ as $n\to\infty$.
\end{remark}

We now proceed with a characterization of the stronger condition \eqref{cond.Mstar} together with \eqref{cond.T}.

\begin{proposition}\label{prop.levy.uniqueness}
    The following statements are equivalent.
\begin{enumerate}
    \item[(i)] The conditions \eqref{cond.Mstar} and \eqref{cond.T} are satisfied.
    \item[(ii)] There exists a unique L\'evy triplet $(b,\si,\nu)$ such that, for all null sequences $(h_n)_{n\in \N}\subset (0,\infty)$ and all $f\in \Cb^2(\R^d)$,
		\[
		\lim_{n\to \infty} \int_{\R^d}\frac{f(\,\cdot+y)-f}{h_n}\,\mu_{h_n}(\d y) =L_{(b,\sigma,\nu)}f\quad \text{in the mixed topology.}
		\]
        \end{enumerate}
\end{proposition}

\begin{proof}
 The implication (ii)$\,\Rightarrow\,$(i) is a direct consequence of Proposition \ref{prop.conv.levy.mixed}.\ In view of Condition \eqref{cond.Mstar} and Proposition \ref{prop.conv.levy.mixed}, the other implication follows once we have shown that, for two L\'evy triplets $(b_1,\si_1,\nu_1)$ and $(b_2,\si_2,\nu_2)$, \begin{equation}\label{eq.levyequality}\big(L_{(b_1,\si_1,\nu_1)}f\big)(0)=\big(L_{(b_2,\si_2,\nu_2)}f\big)(0)\quad \text{for all }f\in \Ccinf(\R^d)\end{equation}
 implies that $(b_1,\si_1,\nu_1)=(b_2,\si_2,\nu_2)$.\ To that end, let $(b_1,\si_1,\nu_1)$ and $(b_2,\si_2,\nu_2)$ be two L\'evy triplets with \eqref{eq.levyequality}.\
 In a first step, we consider a compact set $K\subset \R^d\setminus\{0\}$.\ Since $K$ is compact with $0\notin K$, it follows that $\dist(K,0)>0$, so that, using Friedrichs mollifier, there exist $\de>0$ and a sequence $(\chi_n)_{n\in \N}\subset \Ccinf(\R^d)$ with $\lim_{n\to \infty}\chi_n(x)= \eins_K(x)$ for all $x\in \R^d$, $0\leq \chi_n\leq 1$ for all $n\in \N$, and $\chi_n(x)=0$ for all $n\in \N$ and $x\in \R^d$ with $|x|\leq \de$.\ Then,
\begin{align*}
 \nu_1(K)&=\lim_{n\to \infty}\int_{\R^d\setminus\{0\}}\chi_n(y)\,\nu_1(\d y)=\lim_{n\to \infty}\big(L_{(b_1,\si_1,\nu_1)}\chi_n\big)(0)=\lim_{n\to \infty}\big(L_{(b_2,\si_2,\nu_2)}\chi_n\big)(0)\\
 &=\lim_{n\to \infty}\int_{\R^d\setminus\{0\}}\chi_n(y)\,\nu_2(\d y)=\nu_2(K).
\end{align*}
Since the system of all compacts subsets of $\R^d\setminus\{0\}$ is an intersection-stable generator of the Borel $\si$-algebra on $\R^d\setminus\{0\}$ and both $\nu_1$ and $\nu_2$ are $\si$-finite (as they assign finite measure to compact subsets of $\R^d\setminus\{0\}$), it follows that $\nu_1=\nu_2$.\ Now, one simply considers a function $\chi\in \Ccinf(\R^d)$ with $\nabla \chi(0)=b_1$ and $\nabla^2 \chi(0)=0$ to obtain
\[
|b_1|=\langle b_2,b_1\rangle,
\]
which, by the Cauchy-Schwarz inequality implies that $b_1=b_2$. Last but not least, one chooses a function $\chi\in \Ccinf(\R^d)$ with $\nabla^2 \chi(0)=\si_1$, so that
\[
 \trace\big(\si_1^\top \si_1\big)=\trace\big(\si_2^\top \si_1\big),
\]
which, again by the Cauchy-Schwarz inequality, implies that $\si_1=\si_2$.\ The proof is complete.
\end{proof}

We conclude this section with three examples of families $(\mu_t)_{t>0}$ that satisfy the stronger condition \eqref{cond.Mstar} and \eqref{cond.T}.

\begin{example}[L\'evy processes]\label{ex.levy}
	Let $(Y_t)_{t>0}$ be a $d$-dimensional L\'evy process and $\mu_t$ the distribution of $Y_t$ for $t>0$.\ Then, Condition \eqref{cond.Mstar} and Condition \eqref{cond.T} follow from the well-known L\'evy-Khinchine representation, cf. \cite[Theorem 3.3.3(3), p.\ 164]{applebaum2009levy} or \cite[Theorem 31.5, p.\ 208]{sato1999levy} together with Lemma \ref{lem.apriori.gamma} b) and Lemma \ref{lem.smoothcutoff}. 
\end{example}

\begin{example}[Stochastic convolution]\label{ex.OU}
	Let $A\in \R^{d\times d}$, $b\in \R^d$, $\sigma\in \R^{d\times d}$ be a symmetric and positive semidefinite matrix, and $\mu_t$ be the law of the stochastic convolution
	\[
	Y_t:=\int_0^t e^{(t-s)A}\,(b\,\d t+\sigma\, \d W_t)\quad \text{for }t>0,
	\]
	where $(W_t)_{t\geq 0}$ is a $d$-dimensional standard Brownian motion on some probability space $(\Omega,\mathcal F,\mathbb P)$.\ Then, Condition \eqref{cond.Mstar} and Condition \eqref{cond.T} are satisfied.\ If $\psi_t(x):=e^{tA}x$ for $t> 0$ and $x\in \R^d$, then $p_t(x,B):=\mu_t\big(\big\{y\in \R^d\,|\, \psi_t(x)+y\in B\big\}\big)$, for $t>0$, $x\in \R^d$, and $B\in \Bc(\R^d)$, is the transition kernel of the Ornstein-Uhlenbeck process given by the SDE
	\[
	\d X_t^x=(AX_t^x+b)\,\d t+\sigma\, \d W_t,\quad X_0^x=x\in \R^d.
	\]
\end{example}

\begin{example}[Central Limit Theorem]\label{ex.clt}
 Let $(\Om,\mathcal F,\P)$ be a probability space and $(X_n)_{n\in \N}$ be an i.i.d.\ sequence of random vectors with $\E[X_1]=0$ and $\E[|X_1|^2]<\infty$.\ Moreover, let $\si:=\E[X_1X_1^\top]\in \R^{d\times d}$, and define 
 \[
 \mu_t:=\P\circ \big(\sqrt{t}X_1\big)^{-1}\quad \text{for all }t>0.
 \]
 Then, for all $\chi\in \Ccinf(\R^d)$,
    \begin{align*}
    \int_{\R^d} \frac{\chi(y)-\chi(0)}h\,\mu_h(\d y)&=\E\bigg[\frac{\chi\big(\sqrt h X_1\big)-\chi(0)}h\bigg]=\frac 12\int_0^h\E\big[X_1^\top \nabla^2\chi\big(s\sqrt{h}X_1\big) X_1\big]\,\d s\\
    &\quad\to \trace\big(\sigma \nabla^2\chi(0)\big)\quad \text{as }h\downarrow0,
    \end{align*}
    and, using Markov's inequality,
    \[
    \sup_{t>0}\frac{\P\big(\sqrt{t}X_1>M\big)}{t}\leq \frac{\E\big[|X_1|^2\big]}{M^2}\to 0\quad \text{as }M\to \infty,
    \]
    so that conditions \eqref{cond.Mstar} and \eqref{cond.T} are satisfied.
    Observe that, for all $h>0$ and $k\in \N$,
    \[
     \int_{\R^d} f(y)\, \mu_h^{\ast k}(\d y)=\E\Bigg[f\bigg(\sqrt h\sum_{i=1}^kX_i\bigg)\Bigg]
    \]
    and, if $h=\frac1k$,
    \[
      \int_{\R^d} f(y)\, \mu_h^{\ast k}(\d y)=\E\Bigg[f\bigg(\frac1{\sqrt k}\sum_{i=1}^kX_i\bigg)\Bigg].
    \]
\end{example}

\section{Discussion of Condition (D)} \label{sec.D}

We start this section with several examples of the family $(\psi_t)_{t>0}$ that all satisfy the stronger condition \eqref{cond.Dstar}. 

\begin{example}[Semiflows]\label{ex.diffusion}
	Let $F\colon \R^d\to \R^d$ be globally Lipschitz continuous with Lipschitz constant $L\geq 0$.
	Then, by the Picard-Lindel\"of Theorem, for all $x\in \R^d$, the initial value problem
	\begin{align}\label{eq.ode}
		u'(t)&=F\big(u(t)\big)\quad \text{for all }t\geq 0,\\
		\notag u(0)&=x,
	\end{align}
	has a unique global solution $u\colon [0,\infty)\to \R^d$, which is continuously differentiable, and we define $\psi_t(x):=u(t)$ for all $t>0$.\ Using the Lipschitz continuity of $F$, for all $x,u\in \R^d$ and $t>0$,
	\begin{align*}
		|\psi_t(x+u)-\psi_t(x)-u|&\leq \int_0^t \big|F\big(\psi_s(x+u)\big)-F\big(\psi_s(x)\big)\big|\, \d s\\
		&\leq L|u|t+ \int_0^t L|\psi_s(x+u)-\psi_s(x)-u|\, \d s.
	\end{align*}
	Using Gronwall's lemma, it follows that
	\[
	\sup_{x\in \R^d}|\psi_t(x+u)-\psi_t(x)-u|\leq Lte^{Lt}|u|
	\]
	for all $u\in \R^d$ and $t>0$.\ Therefore, \eqref{eq.biton} is satisfied for any $h_0>0$ with $\omega=Le^{Lh_0}$.\ Moreover, by definition, 
    \[
    \lim_{h\downarrow 0}\frac{\psi_h(x)-x}{h}=F(x) \quad\text{for all }x\in \R^d,
    \]
    so that Condition \eqref{cond.Dstar} is satisfied.
\end{example}

Alternatively, one can consider a family $(\psi_t)_{t> 0}$, which is given in terms of an Euler scheme, as the following example shows.

\begin{example}[Euler method]\label{ex.euler}
	Let $F\colon \R^d\to \R^d$ be globally Lipschitz with Lipschitz constant $L\geq 0$. Consider the explicit Euler method
	\[
	\psi_t(x):=x+tF(x)\quad\text{for all }t>0 \text{ and }x\in \R^d.
	\]
	Then, for all $u\in \R^d$ and $t>0$,
	\[
	\sup_{x\in \R^d}\frac{|\psi_t(x+u)-\psi_t(x)-u|}{t}=\sup_{x\in \R^d}|F(x+u)-F(x)|\leq L |u|,
	\]
	which shows that \eqref{eq.biton} is satisfied for any $h_0>0$ with $\omega=L$.\ Moreover,
    \[
    \frac{\psi_t(x)-x}{t}=F(x)\quad \text{for all }t>0\text{ and }x\in \R^d,
    \]
    so that again the stronger condition \eqref{cond.Dstar} is satisfied.
\end{example}

More generally, one can also consider arbitrary $s$-step Runge-Kutta methods as the following example shows.

\begin{example}[Runge-Kutta methods]\label{ex.runge}
	Let $F\colon \R^d\to \R^d$ be globally Lipschitz with Lipschitz constant $L\geq 0$. For $s\in \N$, we consider the $s$-step Runge-Kutta method, 
	given by the implicit equations
	\begin{equation}\label{eq.rungekutta}
		k_t^i(x)=F\bigg(x+h \sum_{j=1}^s a_{ij} k_t^j(x)\bigg)\quad\text{for all }x\in \R^d\text{ and }i=1,\ldots, s
	\end{equation}
	for sufficiently small $t>0$ and a matrix $a=(a_{ij})_{1\leq i,j\leq s}\in \R^{s\times s}$.\ We first show that \eqref{eq.rungekutta} has a unique solution that satisfies a certain a priori estimate. To that end, for $t>0$, we consider the fixed-point map $G_t\colon \R^d\times \big(\R^d\big)^s\to \big(\R^d\big)^s$, given by
	\[
	G_t^i(x,k):=F\bigg(x+h \sum_{j=1}^s a_{ij} k^j\bigg)\quad\text{for all }x\in \R^d\text{ and }i=1,\ldots, s.
	\]
	Then, for all $t>0$, $x_1,x_2\in \R^d$, and $k_1,k_2\in \big(\R^d\big)^s$,
	\begin{align}
		\notag |G_t(x_1,k_1)-G_t(x_2,k_2)|_1&\leq sL|x_1-x_2| +tL\sum_{j=1}^s |a_{ij}| |k_1^j-k_2^j|\\
		&\leq sL|x_1-x_2| +tL\|a\|_1 |k_1-k_2|_1, \label{eq.apriori.rk} 
	\end{align}
	where
	\[
	\|a\|_1:=\max_{j=1,\ldots, s}\sum_{i=1}^s |a_{ij}|\quad\text{and}\quad |k|_1:=\sum_{i=1}^s |k^i|\quad\text{for all }k\in \big(\R^d\big)^s.
	\]
	Let $t_0>0$ with $t_0L\|a\|_1\leq 1$.\ Then, by Banach's fixed point theorem, for all $x\in \R^d$ and $t\in (0,t_0)$, the map $\big(\R^d\big)^s\to \big(\R^d\big)^s, \,k\mapsto G_t(x,k)$ has a unique fixed point, which we denote by $k_t(x)\in \big(\R^d\big)^s$, and we define
	\[
	\psi_t(x):=x+h\sum_{i=1}^sb_i k_t^i(x)\quad\text{for all }x\in \R^d\text{ and }t\in (0,t_0)
	\]
	with weights $b_1,\ldots, b_s\in \R$ satisfying $\sum_{i=1}^sb_i=1$.\ Moreover, for $t\geq t_0$ and $x\in \R^d$, we set $\psi_t(x):=x$.\ Using Gronwall's lemma together with the estimate \eqref{eq.apriori.rk} and $t_0L\|a\|_1\leq 1$, it follows that
	\[
	\sup_{x\in \R^d}|k_t(x+u)-k_t(x)|_1\leq sLe^{L\|a\|_1 t}t|u|\leq sLet|u|\quad\text{for all }u\in \R^d\text{ and }t\in (0,t_0).
	\]
	The family $(\psi_t)_{t>0}$ thus satisfies \eqref{eq.biton} with $\omega=sLe\max_{i=1,\ldots,s}|b_i|$ for $h_0=t_0$.\footnote{Actually for all $h_0>0$ since we trivially set $\psi_t(x):=x$ for all $t\geq t_0$.}
	
	Moreover, for all $x\in \R^d$ and $h\in (0,t_0)$,
	\[
	|k_t(x)|_1\leq s|F(x)|+tL\sum_{i=1}^s|a_{ij}||k_t^j(x)|\leq s|F(x)|+tL\|a\|_1|k_h(x)|_1,
	\]
	which, by Gronwall's lemma together with $t_0L\|a\|_1\leq 1$, implies that
	\[
	|k_t(x)|_1\leq se^{L\|a\|_1t}|F(x)|\leq se |F(x)|\quad \text{for all }x\in \R^d\text{ and }t\in (0,t_0).
	\]
	Hence, for all $x\in \R^d$,
	\[
	\frac{\psi_h(x)-x}{h}=\sum_{i=1}^s b_iF\bigg(x+h\sum_{j=1}^sa_{ij} k_h^j(x)\bigg)\to  F(x)\quad\text{as }h\downarrow 0,
	\]
	which shows that Condition \eqref{cond.Dstar} is satisfied.
\end{example}

We continue with a series of immediate consequences that can be drawn from Condition \eqref{cond.D}.

\begin{lemma}\label{lem.bitonA2}\
Assume that Condition \eqref{cond.D} is satisfied.
	\begin{enumerate}
    \item[a)] It holds
    \begin{equation}\label{eq.biton.global}
			\sup_{h\in (0,h_0)}\sup_{x\in \R^d}\frac{|\psi_h(x+u)-\psi_h(x)-u|}{h}\leq \omega |u|\quad\text{for all }u\in \R^d.
		\end{equation}    
        In particular,
         \begin{equation}\label{eq.lip.ref}
			|\psi_h(x_1)-\psi_h(x_2)|\leq e^{\om h}|x_1-x_2|
		\end{equation}    
        for all $h\in (0,h_0)$ and $x_1,x_2\in \R^d$.
\item[b)] There exists a constant $C\geq 0$ such that
		\begin{equation}\label{eq.lineargrowthpsi}
			\sup_{h\in (0,h_0)}\frac{|\psi_h(x)-x|}{h}\leq C\big(1+|x|\big)\quad\text{for all }x\in \R^d.
		\end{equation}	
        \item[c)] For all $R\geq0$ and $r>R$, there exists some $\de\in (0,h_0]$ such that $|\psi_{h}(x)|>R$ for all $h\in (0,\de)$ and $x\in \R^d$ with $|x|> r$.
             \item[d)] Let $f\in \Cb^1(\R^d)$ with
	\begin{equation}\label{eq.bound.derivative}
		C_f:=\sup_{x\in \R^d}\big(1+|x|\big)\big|\nabla f(x)\big|<\infty
	\end{equation}	
	and assume that the family $(\mu_t)_{t>0}$ satisfies Condition \eqref{cond.M}.\ Then, 
		\[
		\sup_{t>0}\bigg\|\int_{\R^d}\frac{f\big(\psi_{t}(\,\cdot\,)+y\big)-f(\,\cdot+y)}{t}\,\mu_{t}(\d y)\bigg\|_\infty<\infty.
		\]
	\end{enumerate}
\end{lemma}

\begin{proof}\
	\begin{enumerate}
    		\item[a)] Let $u\in \R^d$ and $n\in \N$ with $|v|\leq \delta$ for $v:=\frac{u}{n}$.\ Then, for all $h\in (0,h_0)$ and $x\in \R^d$,
		\[
		\frac{|\psi_h(x+u)-\psi_h(x)-u|}{h}\leq \sum_{i=1}^n\frac{\big|\psi_h(x+iv)-\psi_h\big(x+(i-1)v\big)-v\big|}{h}\leq \omega n|v|=\omega |u|.
		\]
		Using the triangle inequality and \eqref{eq.biton.global}, it follows that
		\[
			|\psi_h(x_1)-\psi_h(x_2)|\leq (1+\omega h)|x_1-x_2|\leq e^{\omega h}|x_1-x_2|
		\]
		for all $h\in (0,h_0)$ and $x_1,x_2\in \R^d$.
		\item[b)] By Assumption (D),
		\begin{equation}\label{eq.rh}
		\rh:=\sup_{h\in (0,h_0)}\frac{|\psi_{h}(0)|}{h}<\infty.
		\end{equation}
		Hence, by \eqref{eq.biton.global} in part a),
		\[
		\sup_{h\in (0,h_0)}\frac{|\psi_{h}(x)-x|}{h}\leq \max\{\rh,\omega\}\big(1+|x|\big)\quad\text{for all }x\in \R^d.
		\]
        	\item[c)] Let $R\geq 0$, $r>R$, $C\geq 0$ such that \eqref{eq.lineargrowthpsi} is satisfied, and  $\de\in (0,h_0]$ with $C\de\leq \frac{r-R}{1+r}$. Then, using the inverse triangle inequality, for all $h\in (0,\de)$ and $x\in \R^d$ with $|x|> r$,
		\begin{align*}
			|\psi_{h}(x)|&\geq |x|-|\psi_{h}(x)-x|\geq |x|-C\de (1+|x|)=(1+|x|)\big(1-C\de\big)-1\\
			&\geq (1+|x|)\frac{1+R}{1+r}-1>R.
		\end{align*}
        \item[d)] Let $C\geq 0$ such that \eqref{eq.lineargrowthpsi} holds and $\de\in (0,h_0]$ such that
		\[
		 \om\de\leq \frac12\quad \text{and}\quad \big|\psi_{h}(0)\big|\leq \frac14\quad \text{for all }h\in (0,\de).
		\]
		Then, using Lemma \ref{lem.apriori.gamma}, the fundamental theorem of calculus, part a), and part b), for all $x\in \R^d$ and $h\in (0,\de)$,
		\begin{align*}
		  \bigg|\int_{\R^d}& \frac{f\big(\psi_{h}(x)+y\big)-f(x+y)}{h}\,\mu_{h}(\d y)\bigg|\leq 2\|f\|_\infty \frac{\mu_{h}\big(\big\{y\in \R^d\,\big|\, |y|> \tfrac14\big\}\big)}{h}\\
		  &\quad + \bigg|\int_{\{|y|\leq \frac14\}} \int_0^1\bigg\langle\nabla f\Big(x+y+s\big(\psi_{h}(x)-x\big)\Big),\frac{\psi_{h}(x)-x}{h}\bigg\rangle \,\d s\,\mu_{h}(\d y)\bigg|\\
		  &\leq 2\|f\|_\infty c_{\frac14}+C C_f\sup_{|y|\leq\frac14}\sup_{s\in [0,1]}\frac{1+|x|}{1+\big|x+y+s\big(\psi_{h}(x)-x\big)\big|}\\
		  &\leq 2\|f\|_\infty c_{\frac14}+C C_f\sup_{|y|\leq\frac14}\frac{1+|x|}{1+(1-h\om)|x|-|y|-\big|\psi_{h}(0)\big|}\\
		  &\leq 2\|f\|_\infty c_{\frac14}+2C C_f.
		\end{align*}
        Moreover, for $t\geq \de$ and $x\in \R^d$,
        \[
         \bigg|\int_{\R^d} \frac{f\big(\psi_{t}(x)+y\big)-f(x+y)}{t}\,\mu_{t}(\d y)\bigg|\leq \frac{2}{\de}\|f\|_\infty,
        \]
        which completes the proof.
	\end{enumerate}
\end{proof}

\begin{proposition}\label{prop.existence.F}
   Assume that Condition \eqref{cond.D} is satisfied.\ Then, every null sequence in $(0,\infty)$ has a subsequence $(h_n)_{n\in \N}$
 such that, for all $r\geq 0$,
        \begin{equation}\label{eq.uniformderivative}
			\sup_{|x|\leq r} \bigg|\frac{\psi_{h_n}(x)-x}{h_n}-F(x)\bigg|\to 0\quad\text{as }n\to \infty,
		\end{equation}	
        where $F\colon \R^d\to \R^d$ is $\om$-Lipschitz, cf. \ref{eq.Fgloballip}.
\end{proposition}

\begin{proof}
 Let $D\subset \R^d$ be a countable dense subset of $\R^d$.\ Then, using a diagonal argument, there exists a sequence $(h_n)_{n\in \N}\subset (0,h_0)$ with $h_n\to 0$ as $n\to \infty$ such that
		\begin{equation}\label{eq.conv.ondense}
		F(x):=\lim_{n\to \infty}\frac{\psi_{h_n}(x)-x}{h_n}\in \R^d\quad\text{exists for all }x\in D.
		\end{equation}
		By Lemma \ref{lem.bitonA2} a), it follows that
		\[
		| F(x_1)-F(x_2)|= \lim_{n\to \infty}\bigg|\frac{\psi_{h_n}(x_1)-\psi_{h_n}(x_2)-(x_1-x_2)}{h_n}\bigg|\leq \omega|x_1-x_2|\quad\text{for all }x_1,x_2\in D.
		\]
		Since $D$ is dense in $\R^d$, it follows that $F\colon D\to \R^d$ can be uniquely extended to an $\om$-Lipschitz function $F\colon \R^d\to \R^d$, i.e., \eqref{eq.Fgloballip} holds.\
        
        Now, let $r\geq0$ and $\ep>0$.\ Then, there exist $k\in \N$ and $x_1,\ldots,x_k\in D$ such that, for all $x\in \R^d$ with $|x|\leq r$, there exists some $i\in \{1,\ldots, k\}$ with $\omega |x-x_i|<\frac{\ep}{3}$. By \eqref{eq.conv.ondense}, there exists some $n_0\in \N$ such that, for all $n\in \N$ with $n\geq n_0$,
		\[
		\sup_{i=1,\ldots, k}\bigg|\frac{\psi_{h_n}(x_i)-x_i}{h_n}-F(x_i)\bigg|\leq \frac{\ep}{3}.
		\] 
		Hence, by \eqref{eq.biton.global} and \eqref{eq.Fgloballip}, it follows that
		\[
		\sup_{|x|\leq r} \bigg|\frac{\psi_{h_n}(x)-x}{h_n}-F(x)\bigg|< \frac{2\ep}{3}+ \sup_{i=1,\ldots, k}\bigg|\frac{\psi_{h_n}(x_i)-x_i}{h_n}-F(x_i)\bigg|\leq \ep
		\]
		for all $n\in \N$ with $n\geq n_0$.
\end{proof}

We are now ready to identify the generator in the mixed topology of the Koopman semigroup related to the semiflow $(\psi_{h_n})_{n\in \N}$ on a subspace of $\Cb^1(\R^d)$ in a generalized setting.

\begin{proposition}\label{prop.derivative}
Assume that the conditions \eqref{cond.M} and \eqref{cond.D} are satisfied, and consider a null sequence $(h_n)_{n\in \N}\subset (0,\infty)$ and an $\om$-Lipschitz function $F\colon \R^d\to \R^d$ such that \eqref{eq.uniformderivative} holds for all $r\geq 0$.\
	\begin{enumerate}
		\item[a)] Let $f\in \Cb^1(\R^d)$. Then, for all $r\geq 0$,
	\[
	\sup_{|x|\leq r}\bigg|\int_{\R^d}\frac{f\big(\psi_{h_n}(x)+y\big)-f(x+y)}{h_n}\,\mu_{h_n}(\d y)-\big\langle\nabla f(x),F(x)\big\rangle\bigg|\to 0\quad\text{as }n\to\infty.
	\]
    If $f$ satisfies \eqref{eq.bound.derivative}, then
    \[
    \lim_{n\to \infty}\int_{\R^d}\frac{f\big(\psi_{h_n}(\,\cdot\,)+y\big)-f(\,\cdot +y)}{h_n}\,\mu_{h_n}(\d y)\to \big\langle\nabla f,F\big\rangle\quad \text{in the mixed topology.}
    \]
    \item[b)] Let $f\in \Cb^2(\R^d)$ with
    \begin{equation}\label{eq.apriori.mixed}
		\sup_{n\in \N}\bigg\|\int_{\R^d}\frac{f\big(\psi_{h_n}(\,\cdot\,)+y\big)-f}{h_n}\,\mu_{h_n}(\d y)\bigg\|_\infty<\infty.
        \end{equation}
    Then,
    		\begin{equation}\label{eq.bound.derivative.1}
			\sup_{x\in \R^d}\big|\big\langle F(x),\nabla f(x)\big\rangle\big|<\infty.
		\end{equation}
	\item[c)] Assume that also Condition \eqref{cond.T} is satisfied and let $f\in \Cc^1(\R^d)$.\ Then,
    \[
	\lim_{n\to \infty}\bigg\|\int_{\R^d}\frac{f\big(\psi_{h_n}(\,\cdot\,)+y\big)-f(\,\cdot +y)}{h_n}\,\mu_{h_n}(\d y)-\big\langle\nabla f,F\big\rangle\bigg\|_\infty=0.
	\]
 	
	\end{enumerate}
\end{proposition}

\begin{proof}\
	\begin{enumerate}
		\item[a)] Let $f\in \Cb^1(\R^d)$. Then, by the fundamental theorem of infinitesimal calculus, \eqref{eq.M2.weaker}, \eqref{eq.lineargrowthpsi}, and \eqref{eq.uniformderivative}, for all $r\geq0$,
		\begin{align*}
			\sup_{|x|\leq r}&\bigg|\int_{\R^d} \frac{f\big(\psi_{h_n}(x)+y\big)-f(x+y)}{h_n}\,\mu_{h_n}(\d y)-\big\langle\nabla f(x),F(x)\big\rangle\bigg|\\
			&\leq \sup_{|x|\leq r}C(1+|x|)\int_{\R^d}\int_0^1 \Big|\nabla f\Big(x+y+s\big(\psi_{h_n}(x)-x\big)\Big)-\nabla f(x)\Big|\, \d s\,\mu_{h_n}(\d y)\\
			&\quad + \|\nabla f\|_\infty \sup_{|x|\leq r}  \bigg|\frac{\psi_{h_n}(x)-x}{h_n}-F(x)\bigg| 
            \to 0\quad\text{as } n\to\infty.
		\end{align*}
        \item[b)] Let $f\in \Cb^2(\R^d)$ with \eqref{eq.apriori.mixed}.\ Then, by Lemma \ref{lem.apriori.gamma} together with the triangle inequality,
        \[
          \sup_{n\in \N}\sup_{x\in \R^d} \bigg|\int_{\R^d}\frac{f\big(\psi_{h_n}(x)+y\big)-f(x+y)}{h_n}\,\mu_{h_n}(\d y)\bigg|<\infty,
        \]
        so that, by part a), \eqref{eq.bound.derivative.1} holds.
		\item[c)] Let Condition \eqref{cond.T} be satisfied, $f\in \Cc^1(\R^d)$, and $\ep>0$.\ Then, there exist $M\geq0$ and $n_1\in \N$ such that
		\[
		2\|f\|_\infty\frac{\mu_{h_n}\big(\big\{y\in \R^d\,\big|\, |y|>M\big\}\big)}{h_n}\leq \ep\quad\text{for all }n\in \N\text{ with }n\geq n_1.
		\]
		Since $f\in \Cc^1(\R^d)$, there exist some $R\geq 0$ such that $f(x+y)=0$ for all $x,y\in \R^d$ with $|x|>R$ and $|y|\leq M$.\ Let $r>R$.\ Then, by Lemma \ref{lem.bitonA2} c), there exists some $n_2\in \N$ such that $|\psi_{h_n}(x)|>R$ for all $n\in \N$ with $n\geq n_2$ and $x\in \R^d$ with $|x|>r$. Hence, for all $n\in \N$ with $n\geq n_0:=\max\{n_1,n_2\}$ and $x\in \R^d$ with $|x|>r$,
		\begin{align*}
			\bigg|\int_{\R^d}\frac{f\big(\psi_{h_n}(x)+y\big)-f(x+y)}{h_n}\,\mu_{h_n}(\d y)\bigg|&=\bigg|\int_{|y|>M}\frac{f\big(\psi_{h_n}(x)+y\big)-f(x+y)}{h_n}\,\mu_{h_n}(\d y)\bigg|\\
			&\leq 2\|f\|_\infty\frac{\mu_{h_n}\big(\big\{y\in \R^d\,\big|\, |y|>M\big\}\big)}{h_n}\leq \ep.
		\end{align*}
		The statement now follows from part a) together with the fact that $\langle\nabla f,F\rangle\in \Cc(\R^d)$. 
	\end{enumerate}
\end{proof}

\section{Proofs of Section \ref{sec.setup}}\label{sec.proofs.main}

\begin{proof}[Proof of Theorem \ref{thm.main.mehler}]
 Let $(t_n)_{n\in \N}\subset (0,\infty)$ be a null sequence.\ Then, by Proposition \ref{prop.conv.levy.mixed} and Proposition \ref{prop.existence.F}, there exist a subsequence $(h_n)_{n\in \N}$ of $(t_n)_{n\in \N}$, a L\'evy triplet $(b,\si,\nu)$, and an $\om$-Lipschitz function $F\colon \R^d\to \R^d$ such that
 \begin{equation}\label{eq.full.generator}
   \lim_{n\to \infty}\frac{I(h_n)f -f}{h_n}= L_{b,\si,\nu}f+ \langle\nabla f,F\rangle\quad \text{in the mixed top.\ for all }f\in \Ccinf(\R^d).
 \end{equation}
 
 In the first part of proof, we verify the conditions (i)--(v) in \cite[Assumption 2.8]{blessing2022convergence} for the family $(P_{h_n})_{n\in \N}$ with $\kappa=1$ and $X_n=\R^d$ for all $n\in \N$ in order to apply \cite[Theorem 2.9]{blessing2022convergence}. The conditions (i) and (ii) are trivially satisfied in our case. By Lemma \ref{lem.apriori.gamma} and Lemma \ref{lem.smoothcutoff} with $\kappa(s)=1+s$ for all $s\geq 0$, for all $\ep>0$ and $R\geq 0$, there exists a function $\chi\in \Ccinf(\R^d)$ with $0\leq \chi\leq 1$, $\chi(x)=1$ for all $x\in \R^d$ with $|x|\leq R$, and
 \[
 \sup_{x\in \R^d}\big((1+|x|)\big|\nabla \chi(x)\big|+\big|\big(L_{(b,\si,\nu)} \chi\big)(x)\big|\big)\leq \frac\ep2.
 \]
 Hence, by Proposition \ref{prop.derivative} c), there exists some $n_0\in \N$ such that
 \[
   \| P_{h_n}\chi-\chi\|_\infty\leq  \ep h_n\quad \text{for all }n\in \N\text{ with }n\geq n_0,
 \]
 so that Condition (iii) follows from \cite[Lemma 2.13]{blessing2022convergence}.\ Condition (iv) follows from the estimate
 \[
 \big|(P_{h_n}f)(x_1)-(P_{h_n}f)(x_2)\big|\leq \|f\|_{\Lip} \big|\psi_{h_n}(x_1)-\psi_{h_n}(x_2)\big|\leq e^{\om h_n}\|f\|_{\Lip} |x_1-x_2|.
 \]
 for all $n\in \N$, $f\in \Lipb(\R^d)$, and $x_1,x_2\in \R^d$, where we used \eqref{eq.lip.ref}. Moreover, Condition (v) follows from \eqref{eq.full.generator} together with \cite[Remark 2.4]{blessing2022convergence}.
 
 By \cite[Theorem 2.9]{blessing2022convergence}, after potentially passing to a further subsequence, there exists a strongly continuous Feller semigroup $(S_t)_{t\geq 0}$ on $\Cb(\R^d)$ with generator $A\colon D(A)\to\Cb(\R^d)$, cf.\ Definition \ref{def.fellerSG} such that, for all $t\geq 0$, $f\in \Cb(\R^d)$, $(k_n)_{n\in \N}\subset \N$ with $k_nh_n\to t$ as $n\to \infty$, and $(f_n)_{n\in \N}\subset \Cb(\R^d)$ with $f_n\to f$ as $n\to \infty$,
 \begin{equation}\label{eq.semigroup.chernoff} S_tf=\lim_{n\to\infty}I_{h_n}^{k_n}f_{n} \quad\text{in the mixed topology}. \end{equation}
 Moreover, for all $f\in \Cb(\R^d)$ such that $\big(\frac{P_{h_n}f-f}{h_n}\big)_{n\in \N}$ converges in the mixed topology, it holds $f\in D(A)$ with
 \[
    Af = \frac{P_{h_n}f-f}{h_n}\quad \text{in the mixed topology}.
 \]
 The statement now follows from Corollary \ref{cor.comparison} since
    \[
    Af=A_{F,(b,\si,\nu)}f\quad \text{for all }f\in \Ccinf(\R^d).
    \]
\end{proof} 

\begin{proof}[Proof of Corollary \ref{thm.main.levy}]
	The claim directly follows from Theorem \ref{thm.main.mehler} and the definition of convergence in distribution, choosing $\psi_t(x)=x$ for all $t>0$ and $x\in \R^d$, and $f_n=f\in \Cb(\R^d)$ for all $n\in \N$.
\end{proof}	

\begin{proof}[Proof of Theorem \ref{thm.main.mehler.stronger}]
 This is a direct consequence of Corollary \ref{cor.comparison} together with Proposition \ref{prop.levy.uniqueness} the fact that a sequence of real numbers converges is and only if every subsequence has a convergent subsequence with the same limit.
\end{proof}

\begin{proof}[Proof of Corollary \ref{thm.main.mehler.stronger.Lp}]
	This is a direct consequence of Proposition \ref{prop.levygenerator} and the classical Chernoff approximation \cite[Theorem 5.3, p.\ 90]{pazy2012semigroups} together with \cite[Exercise II.1.15(2)]{engel2000one}, using the fact that $\BUC^2(\R^d)$ and $W_p^2(\R^d)$ as a core for the generator of the L\'evy process in the space $\BUC(\R^d)$ and $L_p(\R^d)$, respectively, and Proposition \ref{prop.chernoff.classical}.
\end{proof}

\appendix

\section{Regularization and Convolution}\label{app.C}

In this section, we provide some auxiliary results, which are used in several proofs. We start with a lemma on smooth cut-off functions.

\begin{lemma}\label{lem.smoothcutoff}
	Let $\kappa\colon [0,\infty)\to [0,\infty)$ be measurable and assume that there exist constants $r\geq 0$, $\de>0$, $c>0$, and $C\geq 1$ such that
	\begin{equation}\label{eq.cond-kappa}
	 c \leq \kappa(s)\leq C \kappa(s+t)\quad \text{for all }s\in [r,\infty)\text{ and }t\in [0,\de]
	\end{equation}
	and
	\[
	\int_r^\infty \frac{1}{\kappa(u)}\,\d u=\infty.\footnote{Observe that, by \eqref{eq.cond-kappa}, $\kappa(s)>0$ for all $s\in [r,\infty)$.\ Moreover, the function $1/\kappa$ is measurable on $[r,\infty)$ since $\kappa$ is assumed to be measurable.}
	\]
	Then, for all $\ep>0$ and $R\geq 0$, there exists a function $\chi\in \Ccinf(\R^d)$ with $0\leq \chi\leq 1$, $\chi(x)=1$ for all $x\in \R^d$ with $|x|\leq R$, and
\[
\sup_{x\in \R^d}\big(\kappa(|x|)\big|\nabla \chi(x)\big|+\big|\nabla^2 \chi(x)\big|\big)\leq \ep.
\]
\end{lemma}

\begin{proof}
	By \eqref{eq.cond-kappa}, 
	$\int_r^R \frac{1}{\kappa(u)}\,\d u \leq \frac{R-r}{c}$ for all $R> r$, so that
	\begin{equation}\label{eq.proof.kappa.int}
		\int_R^\infty \frac{1}{\kappa(u)}\,\d u =\infty \quad \text{for all }R> r.
	\end{equation}
	Let $\th\in \Ccinf(\R)$ with $\th\geq 0$, $\int_\R \th(t)\, \d t=1$, and $\supp \th \subseteq [0,n\delta]$ for some $n\in \N$.\ W.l.o.g.\ we may assume that $\int_\R |\theta'(t)|\,\d t\leq c$.\footnote{Otherwise, let $\th_\al(t):=\al\th(\al t)$ for $t\in \R$ and $\al>0$ with $\al \int_\R |\theta'(t)|\,\d t\leq c$.\ Then, $\th_\al\geq 0$, $\int_\R \theta_\alpha(t)\,\d t=1$, $\supp\th_\al=[0,m\de]$ for $m\in \N$ with $m\geq \frac{n}\al$, and $\int_\R |\theta_\alpha'(t)|\,\d t=\al \int_\R |\theta'(t)|\,\d t\leq c$.}\ Moreover, let $\ep>0$ and $R> r$.\ Then, by \eqref{eq.proof.kappa.int}, there exists some constant $M> R+\de$ such that
	\[
	 \int_{R+n\de}^M \frac{1}{\kappa(u)}\,\d u=\frac{C^n}{\ep},
	\]
	and we define 
	$\ga\colon \R\to [0,1]$ by  
	\[
	\ga(s):=\begin{cases}
		1,& s\leq R+n\de,\\
		1-\frac{\ep}{C^n}\int_{R+n\de}^s \frac{1}{\kappa(u)}\,\d u,& s\in (R+n\de ,M],\\
		0,& s>M.
	\end{cases}
	\]
 Let $\chi_0\colon \R\to \R$ be given by
	\[
	\chi_0(s):=\int_{\R} \ga(s+t)\th(t)\, \d t\quad\text{for all }s\in \R.
	\]
	Then, $\chi_0$ belongs to $\Cnt^\infty(\R)$ with $\chi_0(s)=1$ for all $s\in (-\infty, R]$ and $\chi_0(s)=0$ for all $s\geq M$. For $s\in (R,M)$, it follows that
	\begin{align*}
		|\chi_0'(s)|&\leq \frac{\ep}{C^n} \int_0^{n\de} \frac{\th(t)}{\kappa(s+t)}\, \d t\leq\frac{\ep}{\kappa(s)} 
		\int_0^{n\de} \th(t)\, \d t= \frac{\ep}{\kappa(s)} 
	\end{align*}
	and
	\[
	\big|\chi_0''(s)\big|=\frac{\ep}{C^n} \int_0^{n\de} \frac{|\th'(t)|}{\kappa(s+t)}\, \d t\leq \frac{\ep}{c}  \int_0^{n\de} |\th'(t)|\, \d t\leq \ep.
	\]
	Defining $\chi(x):=\chi_0(|x|)$ for $x\in \R^d$, the claim follows.
\end{proof}

The following lemma is a generalization of the classical H\"older inequality for the convolution, tailored to the setting of this paper.

\begin{lemma}\label{lem.convolution}
    Let $\mu$ be a $\sigma$-finite measure on $\Bc(\R^d)$, $g\in L_1(\R^d,\mu)$, and $p\in [1,\infty]$.\ Then, the convolution operator $L_p(\R^d)\to L_p(\R^d),\, f\mapsto \int_{\R^d} f(\,\cdot+ y)g(y)\,\mu(\d y)$ is well-defined and
\begin{equation}\label{eq.interpol.conv}
     \bigg\|\int_{\R^d} f(\,\cdot + y)g(y)\,\mu(\d y)\bigg\|_p\leq \|f\|_p\|g\|_{1,\mu}\quad\text{for all }f\in L_p(\R^d).
    \end{equation}
\end{lemma}

\begin{proof}
 First, note that
 \[
 \int_{\R^d}\int_{\R^d}|f(x+y)g(y)|\,\mu(\d y)\,\d x=\|f\|_\infty\|g\|_{1,\mu}\quad \text{for all }f\in L_\infty(\R^d)
 \]
 and, by Tonelli's theorem and the translation ivariance of the Lebesgue measure,
 \[ \int_{\R^d}\int_{\R^d}|f(x+y)g(y)|\,\mu(\d y)\,\d x=\|f\|_1\|g\|_{1,\mu}\quad \text{for all }f\in L_1(\R^d).
 \]
 Hence, by Fubini's theorem,
 \[
  \int_{\R^d} \big(f_1(\,\cdot+ y)+f_2(\,\cdot+ y)\big)g(y)\,\mu(\d y)=\int_{\R^d} f_1(\,\cdot+ y)g(y)\,\mu(\d y)+\int_{\R^d} f_2(\,\cdot+ y)g(y)\,\mu(\d y)
 \]
 is an element of $L_1(\R^d)+L_\infty(\R^d)$ for all $f_1\in L_1(\R^d)$ and $f_2\in L_\infty(\R^d)$, and the statement follows from the Riesz-Thorin interpolation theorem.
\end{proof}

\section{An Abstract Approximation Result}\label{app.semigroup}

\begin{proposition}\label{prop.chernoff.classical}
    Let $(E,\|\cdot\|)$ be a Banach space, $(S_t)_{t\geq 0}$ be a family of linear operators $E\to E$, $D$ be a dense subset of $E$, $(P_t)_{t\geq0}$ be a family of bounded linear operators $E\to E$ with operator norm $\|P_t\|_{L(E)}\leq 1$ for all $t\geq 0$ such that
    \begin{equation}\label{eq.chernoff_banach}
  \lim_{h\downarrow 0} \frac{P_h u-u}{h}\quad \text{exists for all }u\in D,
    \end{equation}
    and
    \[
         S_tu=\lim_{n\to \infty} P_{h_n}^{k_n}u
        \]
    for all $u\in E$, all null sequences $(h_n)_{n\in \N}\subset (0,\infty)$, and $(k_n)_{n\in \N}\subset \N$ with $k_nh_n\to t\in [0,\infty)$ as $n\to \infty$.\
    Then, for all $u\in E$ such that $\lim_{h\downarrow 0}\frac{P_h u-u}{h}$ exists,
    \[\lim_{t\downarrow 0}\frac{S_t u-u}{t}=\lim_{h\downarrow 0}\frac{P_h u-u}{h}.\]
\end{proposition}

\begin{proof}
  Let $u\in E$ such that $v:=\lim_{h\downarrow 0}\frac{P_h u-u}{h}$ exists, $(t_n)_{n\in \N}\subset (0,\infty)$ with $t_n\to 0$ as $k\to \infty$, and $\ep>0$. Then, for all $n\in \N$, there exists some $k_n\in \N$ such that
  \[
   \big\|P_{h_n}^{k_n}u-S_{t_n}u\big\|\leq \frac{\ep}{3t_n} \quad \text{with}\quad h_n:=\frac{t_n}{k_n}.
  \]
  Moreover, since $D$ is dense in $E$, there exists some $v_0\in D$ with $\|v-v_0\|<\frac\ep{12}$.\ By definition of $v$ and since $v_0\in D$, there exists some $n_0\in \N$ such that
  \[
   \bigg\|\frac{P_{h_n}u-u}{h_n}-v\bigg\|\leq \frac{\ep}3\quad \text{and}\quad \bigg\|\frac{P_{h_n}v-v}{h_n}\bigg\| \leq \frac{\ep}{6t_n}\quad\text{for all }n\in \N\text{ with }n\geq n_0.
  \]
  Hence, for all $k=1,\ldots, k_n$ and $n\in \N$ with $n\geq n_0$,
  \begin{align*}
    \big\|P_{h_n}^{k}v-v\big\|&\leq \frac{\ep}{6}+\big\|P_{h_n}^kv_0-v_0\big\|\leq \frac{\ep}{6}+\sum_{i=1}^{k}\big\|P_{h_n}^iv_0-P_{h_n}^{i-1}v_0\big\|\leq \frac{\ep}{6}+k\big\|P_{h_n}v_0-v_0\big\|\leq \frac\ep3,
  \end{align*}
  so that
    \begin{align*}
       \bigg\|\frac{S_{t_n}u-u}{t_n}-v\bigg\|&\leq \bigg\|\frac{P_{h_n}^{k_n}u-u}{t_n}-v\bigg\|\leq \frac{2\ep}3+ \frac{1}{k_n}\sum_{k=1}^{k_n} \bigg\|\frac{P_{h_n}^ku-P_{h_n}^{k-1}u}{h_n}-P_{h_n}^{k-1}v\bigg\|\\
       &\leq \frac{2\ep}3+ \bigg\|\frac{P_{h_n}u-u}{h_n}-v\bigg\|\leq \ep
    \end{align*}
    for all $n\in \N$ with $n\geq n_0$.\ The proof is complete.
\end{proof}

\section{L\'evy Processes with Drift}\label{app.levydrift}

In this section, we discuss L\'evy processes with drift.\ We follow \cite[Definition 1.6]{sato1999levy} and L\'evy process to have, almost surely, c\`adl\`ag paths.\ The following result is a special case of \cite[Theorem V.6, p.\ 255]{protter}. For the sake of a self-contained exposition and tight a priori estimates, we provide a short proof.

\begin{lemma}\label{lem.strongsolution}
  Let $(Y_t)_{t\geq0}$ be a L\'evy process on some probability space $(\Omega,\mathcal F,\P)$ and $F\colon \R^d\to \R^d$ be an $\omega$-Lipschitz function. Then, for all $x\in \R^d$, the L\'evy SDE
	\begin{equation}\label{eq.levy.sde.app}
	\d X_t^x= F(X_t^x)\,\d t+\d Y_t\quad \text{with}\quad X_0^x=x.
	\end{equation}
    admits a unique strong solution $(X_t^x)_{t\geq 0}$. Moreover,
    \begin{equation}\label{eq.biton.process}
        |X_t^{x+u}-X_t^{x}-u|+|u|\leq e^{\omega t}|u| \quad\P \text{-a.s.\ for all }t\geq 0\text{ and }x,u\in \R^d
    \end{equation}
     and
     \begin{equation}\label{eq.stoch.cont}
      |X_{t}^{x}-x|\leq e^{\omega t}\big|tF(x)+Y_t\big|\quad\P \text{-a.s.\ for all }t\geq 0\text{ and }x\in \R^d.
     \end{equation}
\end{lemma}

\begin{proof}
We start with the existence. To that end, let $N\in \mathcal F$ with $\P(N)=0$ such that for all $w\in \Om\setminus N$, the map $t\mapsto Y_t(w)$ is c\`adl\`ag. For fixed $T>0$ and $L>\om$, let $E$ denote the space of all bounded measurable functions $[0,T]\to \R$ with norm
\[
 \|X\|_E:=\sup_{t\geq 0}e^{-Lt}|X_t|\quad \text{for }X\in E,
\]
and define $G^{x,w}\colon B([0,T])\to B([0,T])$ for all $x\in \R^d$ and $w\in \Om\setminus N$ by
\[
G^{x,w}_t(X):=x+\int_0^t F(X_s)\,\d s+Y_t(w)\quad \text{for all }t\in[0,T]\; x\in \R^d,\text{ and }w\in \Om\setminus N. 
\]
Then, for all $t\in [0,T]$, $x\in \R^d$, and $w\in \Om\setminus N$, 
\begin{align*}
e^{-Lt}\big| G^{x,w}_t(X^1)-G_t^{x,w}(X^2)\big|&\leq e^{-Lt}\int_0^t|F(X_s^1)-F(X_s^2)|\,\d s\leq \|X^1-X^2\|_B e^{-Lt} \int_0^t\om e^{Ls}\,\d s\\
&\leq \frac{\om}L\|X^1-X^2\|_B.
\end{align*}
By Banach's fixed point theorem, for all $x\in \R^d$ and $w\in \Om\setminus N$, there exists a unique $X^{x}(w)\in E$ with $$G^{x,w}_t\big(X^x(w)\big)=X^{x}(w).$$ In particular, for all $x\in \R^d$ and $w\in \Om$, the map $t\mapsto X_t^{x}(w)$ is c\`adl\`ag and, for all $t\in [0,T]$ and $x\in \R^d$, the map $w\mapsto X_t^{x}(w)$ is measurable since, by Banach's fixed point theorem, $X^{x}(w)$ is the pointwise limit of the sequence of random variables $X_t^n(w):=G_t^{x,w}(X^{n-1})$ for $n\in \N$ and $w\in \Om \setminus N$ with $X^0_t:= x$.\ The uniqueness of the fixed point, allows to extend $X^{x}(w)$ from $[0,T]$ to $[0,\infty)$ for all $x\in \R^d$ and $w\in \Om\setminus N$.

Now, let $t\geq 0$, $x,u\in \R^d$, and $w\in \Om\setminus N$.\ Then, using the triangle inequality
\begin{align*}
 |X_t^{x+u}(w)-X_t^{x}(w)-u|+|u|\leq |u|+\int_0^t \om\big(|X_s^{x+u}(w)-X_s^{x}(w)-u|+|u|\big)\, \d s,
\end{align*}
so that, by Gronwall's inequality, \eqref{eq.biton.process} follows. Last but not least, using the triangle inequality,
\[
|X_{t}^{x}(w)-{x}|\leq \big|tF(x)+Y_t(w)\big|+\int_0^t |X_s^{x}(w)-x|\, \d s+|e^{\omega t}|Y_t-Y_s|
\]
for all $t\geq 0$, $x\in \R^d$, and $w\in \Omega\setminus N$, so that \eqref{eq.stoch.cont} follows again from Gronwall's inequality.
\end{proof}

\begin{definition}\label{def.fellerSG}
 We say that a family $(S_t)_{t\geq 0}$ of positive linear operators $\Cb(\R^d)\to \Cb(\R^d)$ with $S_t 1=1$ for all $t\geq 0$ is a strongly continuous Feller semigroup if
\begin{enumerate}
 \item[(i)] $S_0f=f$ for all $f\in \Cb(\R^d)$ and $S_{s+t}=S_sS_t$ for all $s,t\geq 0$.
 \item[(ii)] $\lim_{n\to \infty}(S_{t_n}f_n)(x_n)=(S_{t}f)(x)$ for all sequences $(t_n)_{n\in \N}\subset [0,\infty)$, $(x_n)_{n\in \N}\subset \R^d$, and $(f_n)_{n\in \N}\subset \Cb(\R^d)$ with $t_n\to t\in [0,\infty)$, $x_n\to x\in \R^d$, and $f_n\to f\in \Cb(\R^d)$ in the mixed topology as $n\to \infty$.
\end{enumerate}
In this case, we call $A\colon D(A)\to \Cb(\R^d)$, given by
\[
D(A):=\bigg\{f\in \Cb(\R^d)\,\bigg|\, \lim_{h\downarrow 0}\frac{S_hf-f}h\text{ exists in the mixed topology}\bigg\}
\]
and $Af:=\lim_{h\downarrow 0}\frac{S_hf-f}h$ in the mixed topology for $f\in D(A)$, the generator of $S$.
\end{definition}

\begin{proposition}\label{prop.transition.levydrift}
Let $(Y_t)_{t\geq0}$ be a L\'evy process on some probability space $(\Omega,\mathcal F,\P)$ with L\'evy triplet $(b,\si,\nu)$, $F\colon \R^d\to \R^d$ be an $\omega$-Lipschitz function, and $(X_t^x)_{t\geq 0}$ be the unique strong solution  to \eqref{eq.levy.sde.app} for all $x\in \R^d$. Then,
\begin{equation}\label{eq.def.semigroup.levysde}
   (T_tf)(x):=\E_\P[f(X_t^x)],\quad \text{for all }t\geq 0,\;f\in \Cb(\R^d),\text{ and }x\in \R^d,    
\end{equation}
defines a strongly continuous Feller semigroup $(T_t)_{t\geq 0}$ with generator $B\colon D(B)\to \Cb(\R^d)$.\ Moreover, for all $f\in \Cb^2(\R^d)$ with 
\[
\sup_{x\in \R^d}\big|\big\langle \nabla f(x),F(x)\big\rangle\big|<\infty,
\]
it follows that $f\in D(B)$ with
\[
 Bf=A_{F,(b,\si,\nu)}f.
\]
\end{proposition}

\begin{proof}
 Clearly, $T_0f=f$ and, by \eqref{eq.biton.process}, $T_tf\in \Cb(\R^d)$ for all $f\in \Cb(\R^d)$.\ By \eqref{eq.stoch.cont},
 \[
 \lim_{n\to \infty}\big|\E_\P\big[f(X_{t_n}^{x_n})\big]-f(x)\big|= 0
 \]
 for all sequences $(t_n)_{n\in \N}\subset [0,\infty)$ and $(x_n)_{n\in \N}\subset \R^d$ with with $t_n\to 0$ and $x_n\to x\in \R^d$ as $n\to \infty$.\ Moreover, by \eqref{eq.biton.process}, for all $r,T\geq0$,
 \[
 \sup_{t\in [0,T]}\sup_{|x|\leq r}\P\big(|X_t^x|>M\big)\leq \P\Big(|X_t^0|>M-e^{\om T}r\Big)\to 0\quad \text{as } M\to \infty,
 \]
 so that
  \[
   \lim_{n\to \infty}\sup_{t\in [0,T]}\sup_{|x|\leq r}\big|\E_\P\big[f_n(X_{t}^{x})-f(X_t^x)\big]\big|= 0
  \]
  for all sequences $(f_n)_{n\in \N}\subset \Cb(\R^d)$ with $f_n\to f$ in the mixed topology as $n\to \infty$. Since, by \cite[Theorem V.32, p.\ 300]{protter}, $T_{s+t}=T_sT_t$ for all $s,t\geq 0$, it follows that $(T_t)_{t\geq0}$ is a strongly continuous Feller semigroup.
 
 Now, let $f\in \Cb^2(\R^d)$ with $\sup_{x\in \R^d}|\langle \nabla f(x),F(x)\rangle<\infty$.\ Then, using the fundamental theorem of calculus and Fubini's theorem, 
 \begin{align*}
   \frac{(T_hf)(x)-f(x)}h&=\frac1h\int_0^h\E_\P\big[ \big\langle\nabla f(X_s^x),F(X_s^x)\big\rangle\big]\, \d s+ \frac{\E_\P\big[f(x+Y_h)\big]-f(x)}h.
 \end{align*}
 Since $\langle \nabla f,F\rangle\in \Cb(\R^d)$, it follows that
    \[
     \bigg\|\frac{T_hf-f}h\bigg\|_\infty\leq \|\langle \nabla f,F\rangle\|_\infty+ \bigg\|\frac{\E_\P\big[f(\,\cdot +Y_h)\big]-f}h\bigg\|_\infty<\infty
    \]

 and
 \begin{align*}
    \sup_{|x|\leq r} \bigg|\frac{(T_hf)(x)-f(x)}h-\big\langle \nabla f(x),F(x)\big\rangle -\big(L_{(b,\si,\nu)}f\big)(x)\bigg|\to 0
 \end{align*}
 for all $r\geq 0$, where we used the fact that $(T_t)_{t\geq0}$ is a strongly continuous Feller semigroup and Proposition \ref{prop.conv.levy.mixed}.
\end{proof}

\begin{corollary}\label{cor.comparison}
 Let $(Y_t)_{t\geq0}$ be a L\'evy process on some probability space $(\Omega,\mathcal F,\P)$ with L\'evy triplet $(b,\si,\nu)$, $F\colon \R^d\to \R^d$ be an $\omega$-Lipschitz function, and $(X_t^x)_{t\geq 0}$ be the unique strong solution  to \eqref{eq.levy.sde.app} for all $x\in \R^d$. Then, for every strongly continuous Feller semigroup $(S_t)_{t\geq 0}$ with generator $A\colon D(A)\to \Cb(\R^d)$ satisfying $\Ccinf(\R^d)\subset D(A)$ with
 \[
 Af=A_{F,(b,\si,\nu)}f\quad \text{for all }f\in \Ccinf(\R^d),
 \]
 it follows that
 \[
  (S_tf)(x)=\E_\P[f(X_t^x)],\quad \text{for all }t\geq 0,\;f\in \Cb(\R^d),\text{ and }x\in \R^d.
 \]
\end{corollary}

\begin{proof}
 The statement follows from \cite[Theorem 4.9]{blessing2022convex} once we have verified its assumptions.\ By Proposition \ref{prop.transition.levydrift}, $(T_t)_{t\geq 0}$ and $(S_t)_{t\geq 0}$ are convex monotone semigroups on $\Cb$ in the sense of \cite[Definition 3.1]{blessing2022convex}. Moreover, by \eqref{eq.biton.global}, $(T_t)_{t\geq 0}$ satisfies \cite[Condition (4.7)]{blessing2022convex} and $$\|T_tf\|_{\Lip}\leq e^{\om t}\|f\|_{\Lip}\quad \text{for all }t\geq 0\text{ and }f\in \Lipb(\R^d).$$ Let $(\chi_n)_{n\in \N}\subset \Ccinf(\R^d)$ with $0\leq \chi_n\leq 1$, $\chi_n(x)=1$ for all $x\in \R^d$ with $|x|\leq 1$, and 
 \begin{equation}\label{eq.bound.chin}
 \sup_{x\in \R^d}\big((1+|x|)\big|\nabla \chi_n(x)\big|+\big|\nabla^2 \chi_n(x)\big|\big)\leq 1
 \end{equation}
for all $n\in \N$, see Lemma \ref{lem.smoothcutoff} with $\kappa(s)=1+s$ for all $s\geq 0$, and $f\in \Cbinf(\R^d)$ with $f\geq 0$ and 
 \[
  c_f:=\limsup_{h\downarrow 0}\sup_{x\in \R^d}\frac{(T_hf)(x)-f(x)}h<\infty.
 \]
 Then, $f_n:=f\chi_n\in \Ccinf(\R^d)$ for all $n\in \N$ and, by \cite[Lemma 4.8]{blessing2022convex}, 
 \[
 \sup_{|x|\leq n}(Bf_n)(x)\leq c_f\quad \text{for all }n\in \N.
\]
Moreover, by \eqref{eq.bound.chin} together with the product rule, 
\[
c:=\sup_{n\in \N}\big\|f\langle \nabla \chi_n,F\rangle+ L_{(b,\si,\nu)}f_n\big\|_\infty<\infty,
\]
so that, again by the product rule,
\begin{align*}
 (Bf_n)(x)&=\big(A_{F,(b,\si,\nu)}f_n\big)(x)=\big\langle \nabla f(x),F(x)\big\rangle+f(x)\big\langle \nabla \chi_n(x),F(x)\big\rangle+ \big(L_{(b,\si,\nu)}f_n\big)(x)\\
 &\geq \langle \nabla f(x),F(x)\rangle-c
\end{align*}
for all $n\in \N$ and $x\in \R^d$ with $|x|\leq n$.\
Hence,
\[
\sup_{x\in \R^d}\big\langle \nabla f(x),F(x)\big\rangle\leq c+c_f,
\]
which implies that
\[
(Bf_n)(x)=\chi_n(x)\big\langle \nabla f(x),F(x)\big\rangle+f(x)\big\langle \nabla \chi_n(x),F(x)\big\rangle+ \big(L_{(b,\si,\nu)}f_n\big)\leq 2c+c_f
\]
for all $x\in \R^d$ and $n\in \N$.\ The statement now follows from \cite[Theorem 4.9]{blessing2022convex}.
\end{proof}



\begin{thebibliography}{10}
	
	\bibitem{MR2424078}
	R.~A. Adams and J.~J.~F. Fournier.
	\newblock {\em Sobolev spaces}, volume 140 of {\em Pure and Applied Mathematics
		(Amsterdam)}.
	\newblock Elsevier/Academic Press, Amsterdam, second edition, 2003.
	
	\bibitem{zbMATH01876814}
	A.~Albanese and F.~K{\"u}hnemund.
	\newblock Trotter-kato approximation theorems for locally equicontinuous
	semigroups.
	\newblock {\em Riv. Mat. Univ. Parma (7)}, 1:19--53, 2002.
	
	\bibitem{zbMATH02057702}
	A.~A. Albanese and E.~Mangino.
	\newblock Trotter--kato theorems for bi-continuous semigroups and applications
	to {Feller} semigroups.
	\newblock {\em J. Math. Anal. Appl.}, 289(2):477--492, 2004.
	
	\bibitem{applebaum2009levy}
	D.~Applebaum.
	\newblock {\em L\'{e}vy processes and stochastic calculus}, volume 116 of {\em
		Cambridge Studies in Advanced Mathematics}.
	\newblock Cambridge University Press, Cambridge, second edition, 2009.
	
	\bibitem{blessing2022convex}
	J.~Blessing, R.~Denk, M.~Kupper, and M.~Nendel.
	\newblock Convex monotone semigroups and their generators with respect to
	{$\Gamma $}-convergence.
	\newblock {\em J. Funct. Anal.}, 288(8):Paper No. 110841, 73, 2025.
	
	\bibitem{blessing2022convergence}
	J.~Blessing, M.~Kupper, and M.~Nendel.
	\newblock Convergence of infinitesimal generators and stability of convex
	monotone semigroups.
	\newblock {\em Preprint arXiv:2305.18981}, 2023.
	
	\bibitem{MR2267655}
	V.~I. Bogachev.
	\newblock {\em Measure theory. {V}ol. {I}, {II}}.
	\newblock Springer-Verlag, Berlin, 2007.
	
	\bibitem{zbMATH00897938}
	V.~I. Bogachev, M.~R{\"o}ckner, and B.~Schmuland.
	\newblock Generalized {Mehler} semigroups and applications.
	\newblock {\em Probab. Theory Relat. Fields}, 105(2):193--225, 1996.
	
	\bibitem{zbMATH06871300}
	Y.~A. Butko.
	\newblock Chernoff approximation of subordinate semigroups.
	\newblock {\em Stoch. Dyn.}, 18(3):19, 2018.
	\newblock Id/No 1850021.
	
	\bibitem{zbMATH07238029}
	Y.~A. Butko.
	\newblock The method of {Chernoff} approximation.
	\newblock In {\em Semigroups of operators -- theory and applications. Selected
		papers based on the presentations at the conference, SOTA 2018, Kazimierz
		Dolny, Poland, September 30 -- October 5, 2018. In honour of Jan Kisy\'nski's
		85th birthday}, pages 19--46. Cham: Springer, 2020.
	
	\bibitem{engel2000one}
	K.-J. Engel and R.~Nagel.
	\newblock {\em One-parameter semigroups for linear evolution equations}.
	\newblock Springer, 2000.
	
	\bibitem{haydon}
	D.~H. Fremlin, D.~J.~H. Garling, and R.~G. Haydon.
	\newblock Bounded measures on topological spaces.
	\newblock {\em Proc. London Math. Soc. (3)}, 25:115--136, 1972.
	
	\bibitem{zbMATH01436912}
	M.~Fuhrman and M.~R{\"o}ckner.
	\newblock Generalized {Mehler} semigroups: {The} non-{Gaussian} case.
	\newblock {\em Potential Anal.}, 12(1):1--47, 2000.
	
	\bibitem{kocan}
	B.~Goldys and M.~Kocan.
	\newblock Diffusion semigroups in spaces of continuous functions with mixed
	topology.
	\newblock {\em J. Differential Equations}, 173(1):17--39, 2001.
	
	\bibitem{Goldys2022mixed}
	B.~Goldys, M.~Nendel, and M.~R\"ockner.
	\newblock Operator semigroups in the mixed topology and the infinitesimal
	description of {M}arkov processes.
	\newblock {\em J. Differential Equations}, 412:23--86, 2024.
	
	\bibitem{jacob2001_volume_I}
	N.~Jacob.
	\newblock {\em Pseudo Differential Operators and Markov Processes}.
	\newblock Published by Imperial College Press and distributed by World
	Scientific Publishing Co., 2001.
	
	\bibitem{kuhnemund2001}
	F.~K\"uhnemund.
	\newblock {B}i–{C}ontinuous {S}emigroups on {S}paces with {T}wo {T}opologies:
	{T}heory and {A}pplications.
	\newblock {\em Dissertation}, Universit\"at T\"ubingen, 2001.
	
	\bibitem{zbMATH02111542}
	F.~K{\"u}hnemund and J.~van Neerven.
	\newblock A {Lie}-{Trotter} product formula for {Ornstein}-{Uhlenbeck}
	semigroups in infinite dimensions.
	\newblock {\em J. Evol. Equ.}, 4(1):53--73, 2004.
	
	\bibitem{kunze}
	M.~Kunze.
	\newblock Continuity and equicontinuity of semigroups on norming dual pairs.
	\newblock {\em Semigroup Forum}, 79(3):540--560, 2009.
	
	\bibitem{Nendel2022lsc}
	M.~Nendel.
	\newblock Lower semicontinuity of monotone functionals in the mixed topology on
	{$C_b$}.
	\newblock {\em Finance Stoch.}, 29(1):261--287, 2025.
	
	\bibitem{pazy2012semigroups}
	A.~Pazy.
	\newblock {\em Semigroups of Linear Operators and Applications to Partial
		Differential Equations}, volume~44 of {\em Applied {{Mathematical
				Sciences}}}.
	\newblock {Springer-Verlag, New York}, 1983.
	
	\bibitem{protter}
	P.~E. Protter.
	\newblock {\em Stochastic integration and differential equations}, volume~21 of
	{\em Stochastic Modelling and Applied Probability}.
	\newblock Springer-Verlag, Berlin, second edition, 2005.
	\newblock Corrected third printing.
	
	\bibitem{sato1999levy}
	K.-i. Sato.
	\newblock {\em L\'{e}vy processes and infinitely divisible distributions},
	volume~68 of {\em Cambridge Studies in Advanced Mathematics}.
	\newblock Cambridge University Press, Cambridge, 1999.
	\newblock Translated from the 1990 Japanese original, Revised by the author.
	
	\bibitem{zbMATH01626172}
	G.~Tessitore and J.~Zabczyk.
	\newblock Trotter's formula for transition semigroups.
	\newblock {\em Semigroup Forum}, 63(2):114--126, 2001.
	
	\bibitem{wiweger}
	A.~Wiweger.
	\newblock Linear spaces with mixed topology.
	\newblock {\em Studia Math.}, 20:47--68, 1961.
	
	\bibitem{zbMATH07851043}
	V.~A. Zagrebnov, H.~Neidhardt, and T.~Ichinose.
	\newblock {\em Trotter-Kato product formul{{\ae}}}, volume 296 of {\em Oper.
		Theory: Adv. Appl.}
	\newblock Cham: Birkh{\"a}user, 2024.
	
\end{thebibliography}

\end{document}